\documentclass[12pt,oneside]{amsart}
\usepackage{geometry}                
\usepackage{graphicx}
\usepackage{amssymb}
\usepackage{epstopdf}
\usepackage{float}
\usepackage{wrapfig}
\usepackage{comment}
\usepackage{fourier}

\usepackage{amsmath,amsthm,amscd,amssymb}
\usepackage{latexsym}
\usepackage[colorlinks,citecolor=red,pagebackref,hypertexnames=false]{hyperref}
\numberwithin{equation}{section}
\theoremstyle{plain}
\newtheorem{theorem}{Theorem}[section]
\newtheorem{lemma}[theorem]{Lemma}
\newtheorem{corollary}[theorem]{Corollary}
\newtheorem{proposition}[theorem]{Proposition}

\theoremstyle{definition}
\newtheorem{definition}[theorem]{Definition}

\theoremstyle{remark}
\newtheorem{fact}[theorem]{Fact}
\newtheorem{remark}[theorem]{Remark}

\newtheorem{case[theorem]}{Case}

\newcommand*{\setdiff}{\bigtriangleup}

\usepackage{hyperref}
\usepackage{color}
\definecolor{blue}{rgb}{0,0,1}

\definecolor{red}{rgb}{1,0,.2}

\usepackage{soul,xcolor}

\newcommand{\ybox}{}

\date{{\today}}
\author{K\'aroly Simon}
\address[K\'aroly Simon]{Budapest University of Technology and Economics, Department of Stochastics, Institute of Mathematics, and MTA-BME Stochastics Research
Group, 1521 Budapest, P.O.Box 91, Hungary} \email{simonk@math.bme.hu}

\author{Krystal Taylor }

\address[Krystal Taylor]{Department of Mathematics, The Ohio State, Columbus, OH}
\email{taylor.2952@osu.edu}

\thanks{This work came out of a collaboration that started at ICERM at Brown University in Rhode Island. Also this research was supported by the Mathematics Research Institute of the Ohio State University.The research of K. Simon was partially supported by the grant OTKA K104745,
by MTA-BME Stochastics Research
Group and
by ICERM by supporting his participation on one of their a semester programs in 2016.}

\title{\parbox{14cm}{\centering{ Interior of Sums of Planar sets and Curves  }}}

 \usepackage{xr}
 \usepackage{xr}
\begin{document}

\setstcolor{red}

\maketitle

\begin{abstract}
Recently, considerable attention has been given to the study of the arithmetic sum
of two planar sets.
We focus on understanding the interior $\left(A+\Gamma\right)^{\circ}$,
when $\Gamma$ is a piecewise $\mathcal{C}^2$ curve and $A\subset \mathbb{R}^2.$
To begin, we give an example of a  very large (full-measure, dense, $G_\delta$) set $A$
such that $\left(A+S^1\right)^{\circ}=\emptyset$, where $S^1$ denotes the unit circle.
This suggests that merely the size of $A$ does not guarantee that $\left(A+S^1\right)^{\circ }\ne\emptyset$.
If, however, we assume that $A$ is
a kind of generalized product of two reasonably large sets,
 then
$\left(A+\Gamma\right)^{\circ}\ne\emptyset$ whenever $\Gamma$ has non-vanishing curvature.
As a  byproduct of our method, we prove that the pinned distance set  of $C:=C_{\gamma}\times C_{\gamma}$, $\gamma \geq \frac{1}{3}$, pinned at any point of $C$ has non-empty interior, where $C_{\gamma}$ (see \eqref{g59}) is the middle $1-2\gamma$ Cantor set (including the usual middle-third Cantor set, $C_{1/3}$). Our proof for the middle-third Cantor set requires a separate method.
We also prove that $C+S^1$ has non-empty interior.
\end{abstract}


\section{Introduction}

Given a  set $A\subset \mathbb{R}^2$, we study the set of points which are at a distance $1$ from at least one of the elements of $A$, where ``distance" refers to either the Euclidean distance or some other natural distance on the plane.
This set is $A+S^1$, where $S^1$ is the unit circle in the given distance.  In fact, we consider piecewise $\mathcal{C}^2$ curves $\Gamma$ in addition to $S^1$, and we investigate
conditions which guarantee that the interior of $A+\Gamma$ is not empty.
\\

This paper is the continuation of our joint paper \cite{STmeas} where we determine the Hausdorff dimension and the positivity of the Lebesgue measure of the sum set $A+\Gamma$.
In particular, a simple Fourier analytic proof shows that if $A\subset \mathbb{R}^2$ and $\dim_{\rm H} (A)>1$, then the \textit{measure} of $A+\Gamma $ is positive, where $\Gamma $ is an arbitrary $C^2$ curve with at least one point of curvature.
Under the same hypotheses on $\Gamma$, it is shown in \cite{STmeas} that
if $\dim_{\rm H} A \leq 1$, then $\dim_{\rm H} (A+\Gamma)=1+\dim_{\rm H} A$.
When $\dim_{\rm H} A = 1$, we prove that $\mathcal{L}^2(A+\Gamma)>0$
if and only if $A$ contains a regular (rectifiable) set of positive $\mathcal{H}^1$-measure.
(See \cite[Theorem 2.1]{STmeas}).
\\

When $\Gamma= S^1$, the set $A+\Gamma$ is the union of circles of radius one with centers in a set $A$.
The positivity of the Lebesgue measure of large unions of circles was also investigated by Marstrand \cite{Mar87}
and Wolff \cite{W00} (also see Oberlin \cite{O06} and Mitsis \cite{M98} for a higher dimensional analogue).
Let $S(a,r)$ denote the circle in the plane with center $a$ and radius $r$, and identify the set of all such circles with $\mathcal{S}=\mathbb{R}^2\times (0,\infty)$.  Given a collection of circles $E\subset \mathcal{S}$
 with dimension greater than $1$, it is reasonable to hypothesize that since a given circle has dimension $1$, then the union over circles in $E$ has dimension $2$.

 \begin{theorem}[Wolff $d=2$]
Given $E\subset \mathbb{R}^2\times \mathbb{R}^+$ satisfying $\ybox{\dim_{\rm H} E>1}$.
Then
 \begin{equation*}\label{O99}
   \mathcal{L}^2 \left( \bigcup_{(\mathbf{a},\rho)\in E} S(a,\rho) \right)\ne 0,
    \end{equation*}
where $  \mathcal{L}^2 (\cdot)$ denotes the Lebesgue measure.
  \end{theorem}
 \smallskip

 As far as we know, no non-trivial results are known on \textit{the interior} of $A+S^1$. In general, it turns out that even positivity of the Lebesgue measure of $A$ is not enough to guarantee that $A+S^1$ has non-empty interior (see our counter example in section \ref{Giant}). In Theorem \ref{main}, we provide conditions on $A\subset \mathbb{R}^2$ which guarantee that the interior of the sum set $\left(A+\Gamma \right)^{\circ}$ is non-empty, where $\Gamma $ is an arbitrary $C^2$ curve with at least one point of curvature.
 \\

Our method involves introducing a $1$-parameter family of Lipschitz maps $\left\{\Phi_\alpha\right\}_{\alpha\in J}$, where $\Phi_\alpha:A\to \ell _\alpha$, $\ell _\alpha$ is the vertical line at $x=\alpha$, and $J$ is an interval.
This family $\left\{\Phi_\alpha\right\}$ is defined in such a way that
\begin{itemize}
  \item the $\Phi_\alpha$-images of the set $A$ are each contained in $A+\Gamma$.
  \item the $\Phi_\alpha$-images of the set $A$ each contain an interval $I$ which is uniform over an interval worth of $\alpha$s.

\end{itemize}

Before we state our main results, we collect some of the most important notation:
\begin{definition}\label{g60}
  \begin{enumerate}
    \item Let $A \subset \mathbb{R}^2$, then $A^\circ$ is the interior of $A$.
    \item A  Cantor set is a totally disconnected perfect set (where perfect refers to a compact set which is identical to its accumulation points).
    \item
    For $\gamma\in(0,1)$, a Symmetric Cantor sets $C_\gamma\subset [0,1]$ (see \cite[Section 8.1]{Mat15}) is defined as follows:
We iterate the same process that yields the usual middle-third Cantor set with the difference that we remove the middle-$1-2\gamma$ portion of every interval throughout the construction:
\begin{equation}\label{g59}
  C_\gamma=\left\{(1-\gamma)\sum\limits_{k=1}^{\infty }
  a_k\gamma^{k-1}:a_k\in\left\{0,1\right\}
  \right\}.
\end{equation}
The so-called middle-$d$ Cantor set is $C_{1-2d}$. In particular $C_{1/3}$ is the usual middle-third Cantor set.
    \item We write
    \begin{equation}\label{g58}
      C(\gamma):=C_\gamma\times C_\gamma.
    \end{equation}
    In particular, $C(1/4)$ is called the \emph{four-corner Cantor set}.
  \end{enumerate}
\end{definition}

\subsection{Summary of main results}
The behavior of $A+\Gamma$ may be conspicuously different when  the piecewise-$\mathcal{C}^2$ curve $\Gamma$ has non-vanishing curvature (this case is considered in Section \ref{g64'}) and when $\Gamma$
is a polygon (this case is studied in Section \ref{w99}).

\subsubsection{Main results when $\Gamma$ has non-vanishing curvature}

  \begin{enumerate}
    \item There exists a set of full $\mathcal{L}^2$-measure which is dense in $\mathbb{R}^2$ and $G_\delta$  (we call it $G$ for Giant) such that $\left(G+S^1\right)^{\circ}=\emptyset$. (See Section \ref{Giant}.)
    \item However, in Theorem \ref{main} we prove that    $(C+\Gamma)^{\circ}\ne\emptyset$ if $C$ has  a kind of generalized product structure  $C=P(A,B)$ (see \eqref{N53} for the definition of $P$) of sets $A$ and $B$ satisfying one of the following properties:\\
   (i) $A$ and $B$ are sets of positive Lebesgue measure.
\\
(ii) $A$ and $B$ are Cantor sets so that $\tau(A)\cdot \tau(B) >1$,
where $\tau(A)$ and $\tau(B)$ stands for the Newhouse thickness (see section \ref{N43})
of $A$ and $B$.
\\
(iii) $A$ and $B$ are sets of second category and $B$ is a Baire set.
\\
(iv) There exists a non-degenerate  interval $J_1$ such that $\#(J_1\setminus A)=\aleph_0$  and
$\# B=\aleph$.
  \item As a byproduct of our proof of the previous Theorem, in Corollary \ref{N41}, we obtain that for every $\frac{1}{3} \leq \gamma <1$ and for every $\mathbf{t}\in C(\gamma)$, the interior of the  pinned distance set at $\mathbf{t}$, $\left(\Delta_{\mathbf{t}}(C(\gamma)\right)^{\circ}\ne\emptyset$, where $C(\gamma)$ was defined  in \eqref{g58}.

\item As an extension of item (2) part (ii) above, we prove that for the middle-third Cantor set $C$
we have $\left((C\times C)+S^1\right)^{\circ} \ne\emptyset $ (Theorem \ref{k1}).

\item In Theorem \ref{annulus},  we prove that $(A\times B)+S^1$ contains an annulus of radius $1$
whenever $A,B \subset \mathbb{R}$ with $\mathcal{L}^1(A)>0$ and $\mathcal{L}^1(B)>0$.
      \item In Lemma \ref{g99} we point out that $(A+\Gamma)^{\circ}\ne\emptyset$ if $A \subset \mathbb{R}^2$ is a connected set.
       \end{enumerate}

\subsubsection{Main results when $\Gamma$ is a polygon}
Let $N_\theta$ be the angle $\theta$ rotation of the perimeter of $[0,1]^2$.
\begin{enumerate}
  \item For every polygon $\Gamma$ we can find a set of full measure
  $A$ such that $(A+\Gamma)^{\circ}=\emptyset$.
  (Theorem \ref{g92}.)
  \item There is set $A$ of positive $\mathcal{L}^2$-measure such that simultaneously for all $\theta$ we have $(A+N_\theta)^{\circ}=\emptyset$. (Theorem \ref{g71}.)
\end{enumerate}

\section{The statements of our results}\label{g64}

\subsection{The case when $\Gamma$ has non-vanishing curvature}\label{g64'}

\subsubsection{Generalized product structure}

\begin{definition}\label{wavy}
Let $A,B\subset\mathbb{R}$ and $z\in \mathcal{C}^2(\mathbb{R}^2)$. We define the set $P(A,B)$
\begin{equation}\label{N53}
  P(A,B):=P_z(A,B):=
\left\{
(x,z(x,y)):x\in A\mbox{ and } y\in B
\right\},
\end{equation}
where
\begin{equation}\label{a01}
 z_y(x,y)\ne 0  \mbox{ for each } (x,y) \neq (0,0).
 \end{equation}

Notice that when $z(x,y) = y$, then $P(A,B)= A\times B$.
\end{definition}


The main result in this case is as follows:

\begin{theorem}\label{main}
Let $P(A,B)$ as in \eqref{N53} where $A, B \subset \mathbb{R}$ are bounded  sets so that at least one of the following hold:\\
(i) $A$ and $B$ are sets of positive Lebesgue measure.
\\
(ii) $A$ and $B$ are Cantor sets so that $\tau(A)\cdot \tau(B) >1$,
where $\tau(A)$ and $\tau(B)$ stands for the Newhouse thickness (see section \ref{N43})
of $A$ and $B$.
\\
(iii) $A$ and $B$ are sets of second category and $B$ is a Baire set.
\\
(iv) There exists a non-degenerate  interval $J_1$ such that $\#(J_1\setminus A)=\aleph_0$  and
$\# B=\aleph$.\\
Then the interior of the algebraic sum,
\begin{equation}\label{interior} \left( P(A,B) + \Gamma \right)^0 \neq \emptyset,\end{equation}
 is non-empty, where $\Gamma$
 is a set which contains
a $C^2$ sub-curve with non-vanishing curvature.
\end{theorem}

The proof is given in Section \ref{a02}.  The definition of Newhouse thickness is given in section \ref{N43}, and a review of Baire sets and sets of second category appears in section \ref{N42Baire}.

\begin{remark}
  This result shows in the simplest case that $(A\times B)+S^1$ has non empty interior if any of the conditions $(i)$-$(iv)$ hold.
\end{remark}

\begin{remark} More generally, if $\Gamma = \{(x,y): x^p + y^p =1\},$ then $$\left((A\times B) + \Gamma \right)^0\neq \emptyset$$
provided any of the conditions $(i)$-$(iv)$ hold.
\end{remark}

\begin{remark}\label{w98}
  It is easy to see that all the assumption of Theorem \ref{main} imply that
  $(A+B)^{\circ}\ne \emptyset $. However, there exist $A,B \subset \mathbb{R}$,
  $(A+B)^{\circ}\ne \emptyset$ but $((A\times B)+S^1)^{\circ}= \emptyset $. Namely, consider the sets $A,B \subset \mathbb{R}$ constructed in \cite[Theorem 5.11]{Falc86}. They satisfy:
  \begin{equation}\label{w97}
    \dim_{\rm H} A=\dim_{\rm H} B=0 \mbox{ and }
    A+B \mbox{ is an interval.}
  \end{equation}
On the other hand $A\times B$ is a $1$-set
(its one-dimensional Hausdorff measure is positive and finite). Then
$A\times B$ must be an irregular $1$-set (for the definition see \cite[Section 2]{Falc86}). This is immediate from \eqref{w97} and
\cite[Corollary 6.14]{Falc86}. It follows from
\cite[Corollary 2.3]{STmeas} that an irregular $1$-set plus $S^1$ has Lebesgue measure zero, so it cannot contain interior points.
\end{remark}

In fact, we can say something about the topology of the sum set in the following setting:

\begin{theorem}\label{annulus} Let $A$ and $B$ be subsets of \,$\mathbb{R}$ of positive Lebesgue measure.  Then
$$\left(A\times B \right)+S^1$$ contains an annulus of radius $1$.
\end{theorem}

The proof of Theorem \ref{annulus} appears in Section \ref{annulusproof}.
\\

We also prove the following which can be seen as a supplement to Theorem \ref{main} part (ii).
\begin{theorem}\label{k1}
Let
 $C(1/3)$ denote the product of the middle third $C_{1/3}$ with itself (see \eqref{g58} above).
Then
\begin{equation}\label{k1statement}\left( C(1/3) + S^1\right)^{\circ}\neq \emptyset,\end{equation}
where $S^1$ is the unit circle in the plane.
\end{theorem}

See Section \ref{sectionproofk1} for the proof of Theorem \ref{k1}.
\\

\begin{remark} What we know about the set $C(\gamma)+S^1$  is as follows:

\begin{enumerate}
  \item if $\gamma < \frac{1}{4}$ then $\dim_{\rm H} \left(C(\gamma)+S^1\right)=1-2\frac{\log 2}{\log \gamma}$. (See \cite[Theorem 2.1]{STmeas}).
  \item if $\gamma=\frac{1}{4}$ then $\dim_{\rm H} \left(C(\gamma)+S^1\right)=2$ but $\mathcal{L}^2\left(C(\gamma)+S^1\right)=0$.
      \cite[Corollary 2.3]{STmeas}).
  \item if $\gamma>\frac{1}{4}$ then $\mathcal{L}^2\left(C(\gamma)+S^1\right)>0$. (See \cite[Corollary 3]{W00} and also see \cite[Theorem 2.1]{STmeas}).
  \item if $\gamma \geq \frac{1}{3}$ then
  $\left(C(\gamma)+S^1\right)^{\circ}\ne \emptyset $ (Theorem \ref{k1} ).
\end{enumerate}
We do not know if there are  $\gamma\in\left(\frac{1}{4},\frac{1}{3}\right)$ with
$\left(C(\gamma)+S^1\right)^{\circ}\ne \emptyset $.
\end{remark}

The proof of Theorem \ref{main} relies on verifying the following proposition.
Proposition \ref{T50} is a strengthening of a classic theorem of Steinhaus \cite{St20, S72} on the interior of difference sets.  Moreover,  Proposition \ref{T50} improves on a result of Erd\H{o}s and Oxtoby on more general difference sets and provides an alternative proof to their main theorem in  \cite{EO54} .

\begin{proposition}\label{T50}
Assume that $J_1,J_2$ are compact intervals on $\mathbb{R}$, and let $\Lambda$ be a parameter interval (a non-empty open interval).
 Let $A\subset J_1$ and $B\subset J_2$.
We assume that the pair of sets $A,B$ satisfy  at least one
 of the following conditions:\\
\noindent (i) $A$ and $B$ are sets of positive Lebesgue measure,\\
(ii) $A$ and $B$ are Cantor sets so that $\tau(A)\cdot \tau(B) >1$,\\
(iii) $A$ and $B$ are sets of second category and $B$ is a Baire set. \\
(iv) $\#(J_1\setminus A)=\aleph_0$  and
$\# B=\aleph$.\\
Let
   $H(\alpha,x,y)\in \mathcal{C}^2(\Lambda\times J_1\times J_2)$ be an arbitrary function  with non-vanishing partial derivatives in $x$ and $y$ on $\Lambda\times J_1\times J_2$. Then  there exists a non-empty open interval $I\subset \Lambda$ sharing the same center as $\Lambda$
such that the interior
 \begin{equation*}
 \left( \bigcap_{\alpha\in I}H(\alpha, A, B)\right)^0\neq \emptyset.
  \end{equation*}
\end{proposition}
\vskip.125in

For proofs and background, see section \ref{P50}.
\begin{remark}
Only part (ii) of Proposition \ref{T50} requires the stronger assumption of
  $H(\alpha,x,y)\in\mathcal{C}^2(\Lambda\times J_1\times J_2)$.   Otherwise $\mathcal{C}^1$ is enough.
\end{remark}

\begin{remark}
In the case that $H(\alpha, x,y) = x+y$, part $(i)$ of Proposition \ref{T50} implies the classic theorem of Steinhaus \cite{St20, S72} on the interior of difference sets.  When $H(\alpha, x,y) = H(x,y)$, it implies the result proved by Erd\H{o}s and Oxtoby in \cite{EO54}.
\end{remark}


\subsubsection{Pinned distance sets}\label{g89}

Proposition \ref{T50} yields an interesting consequence for pinned distance sets.
The celebrated Falconer distance conjecture (see e.g. \cite{Falc86}, \cite{Mat95})
asks how large a subset $E$ of Euclidean space needs to be in order to guarantee that its distance set, defined by $\Delta(E)=\{|x-y|: x,y \in E\}$, has positive Lebesgue measure.  In 1986, Falconer proved  $\dim_{\rm H} E>\frac{d+1}{2}$ suffices and that $\dim_{\rm H} E>\frac{d}{2}$ is necessary where $E \subset {\Bbb R}^d$ and $d \ge 2$.
The best known exponent is $\frac{d}{2}+\frac{1}{3}$ and is due to Wolff when $d=2$ \cite{W04} and Erdogan \cite{Erd05} for $d\geq3$.
In \cite{IMT12},  Iosevich, Mourgoglou and Taylor
study the interior of the distance set and prove Falconer's result for more general notions of distance.
\\

Another interesting variant of the Falconer distance problem is obtained by pinning the distance set. More precisely, given $x \in E$, let $\Delta_x(E)=\{|x-y|: y \in E\}$.
Peres and Schlag (\cite{PeSc00}) proved that for many values of $x$, the Lebesgue measure of the pinned distance set $\Delta_x(E)$ is positive whenever $E$ is a set of Hausdorff dimension greater than $1$.
In \cite{IoTaUr}, Iosevich, Taylor, and Uriarte-Tuero
give a simple proof of the Peres-Schlag result and generalize it to a wide range of distance type functions as well as obtaining a variant for more general geometric configurations.
\\

A consequence of Proposition \ref{T50} for the interior of pinned distance sets
is as follows:

\begin{corollary}\label{N41}
  Let $C\subset \mathbb{R}$ be a Cantor set, and let $\alpha >1$. We consider the
  the  pinned distance set at $\mathbf{t}$
  with respect to the $\alpha$-norm:
 \begin{equation*}
   \Delta^{(\alpha)}_\mathbf{t}(C\times C):=
   \left\{
   \|\mathbf{c}-\mathbf{t}\|_\alpha:
   \mathbf{c}\in C\times C
   \right\},
 \end{equation*}
 where $\|(x,y)\|_\alpha=\left(|x|^\alpha+|y|^\alpha\right)^{1/\alpha}$.
If the   thickness $\tau(C)>1$, then for every  $\mathbf{t}\in\mathbb{R}^2$ the interior
\begin{equation*}
  \left(\Delta^{(\alpha)}_\mathbf{t}(C\times C)\right)^{\circ}\ne\emptyset.
\end{equation*}
Moreover there exists a non-empty open  parameter interval $I$ centered at $\alpha$
such that
\begin{equation}\label{N39}
   \left(\bigcap_{\beta\in I}\Delta^{(\beta)}_\mathbf{t}(C\times C)\right)^{\circ}\ne\emptyset.
\end{equation}
\end{corollary}

\begin{remark}\label{g62}
 It is easy to see that whenever $\frac{1}{3}<\gamma<1$ the thickness
 $\tau(C_\gamma)>1$ ($C_\gamma$ was definition \ref{g60}). Hence, by Corollary \ref{N41}, the pinned distance set $\Delta_{\mathbf{t}}(C(\gamma))$ contains an interval for every $\mathbf{t}$, (recall that $C(\gamma)$ was defined in Definition \ref{g60} ).
\end{remark}

\begin{remark} Note that Corollary \ref{N41} still holds if $K\times K$ is replaced by $K_1\times K_2$, where $K_1,K_2$ are Cantor sets of sufficient thickness that is $\tau(K_1)\cdot \tau(K_2)>1$.
\end{remark}

\begin{proof}[Proof of Corollary \ref{N41}]
  We fix an arbitrary $\mathbf{t}=(t_1,t_2)\in\mathbb{R}^2$ and
   $\alpha >1$ and choose a parameter interval $\Lambda$ centered at $\alpha$ such that $1 \not\in\Lambda$.
  Let
  $$
  H(\beta,x,t):=(x-t_1)^\beta+(y-t_2)^\beta ,\quad \beta\in\Lambda.
  $$
 Then \eqref{N39} immediately follows from the second part of Proposition \ref{T50}.
\end{proof}

Actually we can prove an analogous theorem for the middle third Cantor set $C_{1/3}$ with completely different technique:

\begin{theorem}\label{g61} Let $C_{1/3}$ be the middle-third Cantor set.
Then the pinned distance set of $C(1/3):=C_{1/3}\times C_{1/3}$,
  $$
  \Delta_{\mathbf{t}}(C({1/3}))=
  \left\{
  \|\mathbf{c}-\mathbf{t}\|_2
  :
  \mathbf{c}\in C(1/3)
  \right\},
  $$
   pinned at an arbitrary $\mathbf{t}\in C(1/3)$, has non-empty interior.
\end{theorem}

The proof is presented in Section \ref{g58'}.

We remark that in \cite{RaSi} the first author and M. Rams investigated a well-studied family of random cantor sets, called "Mandelbrot percolation Cantor sets" on the plane.
 They obtained that in that family,
if the dimension of the attractor is greater than one, then
 almost surely  the pinned distance set pinned at any points of the plane,   contains intervals.

We also remark that in a recent preprint, P. Shmerkin \cite{She17} proved that the Hausdorff dimension of the pinned distance set $\Delta_t(E)$ is equal to $1$ for most elements of $E$ (in a natural sense) if $\dim_{\rm H} (E)>1$ and the packing and Hausdorff dimensions of $E$ are equal. (These assumptions clearly hold for $C(1/3)$.

\subsubsection{$(A+S^1)^{\circ}\ne\emptyset$ for $A$ connected}

As a corollary of part (iv) of Theorem \ref{main} we can easily obtain that

\begin{corollary}\label{a098}
  Let $K$ be the Knaster-Kuratowski fan (or Cantor's teepee), defined below.
  Then
  \begin{equation*}
    \left(K+S^1\right)^{\circ}\ne\emptyset.
  \end{equation*}
\end{corollary}
The Knaster-Kuratowski fan was defined in  \cite[p. 145]{Stee78}
\begin{definition}[Kastner-Kuratowski fan]
Let $C$ be the translated copy of the middle-third Cantor $C_{1/3}$ set situated on the interval $I:=\left\{\left(\frac{1}{2},x\right):x\in\left[-\frac{1}{2},\frac{1}{2}\right]\right\}$.
Denote the set of the elements of $C$ which are deleted interval end points by $E$.
Let  $F:=C\setminus E$. For a $c\in C$ let $\ell _c$ be the line segment which connects the origin with $c$. That is
$$
\ell _c:=\left\{(x,cx):x\in[0,1]\right\}
$$
 We introduce
$$
K_E:=\bigcup_{c\in E}\left\{(x,y)\in \ell_c:x\in \mathbb{Q}
\right\},\mbox{  and }
K_F:=\bigcup_{c\in F}\left\{(x,y)\in \ell_c:x\not\in \mathbb{Q}
\right\}.
$$
The set $K:=K_E\cup K_F$ is called Knaster-Kuratowski fan, Cantor's teepee.
\end{definition}
It was proved in \cite[p. 146]{Stee78} that Knaster-Kuratowski fan $K$ satisfies:
\begin{description}
  \item[(a)] $K$ is connected,
  \item[(b)] $K\setminus (0,0)$ is totally disconnected (all connected components are singletons). As a consequence,
  \item[(c)] $K$ does not contain any paths (continuous image of $[0,1]$ which is non-constant).
\end{description}
\begin{proof}[Proof of Corollary \ref{a098}]
  Observe that $X_F=P([0,1]\setminus\mathbb{Q},C)$, where $P$ is as in \eqref{wavy}
  with $z(x,y):=x \cdot y$. Clearly, $z_y(x,y)$ is not zero if we are off the origin.  So, we can apply part (iv) of Theorem \ref{main} for such a part of $X_F$.
  \end{proof}

  We could prove that $\left(K+S^1\right)^{\circ}\ne\emptyset$ directly from it connectedness  and the following lemma:

  \begin{lemma}\label{g99}
     Assume that $A\subset\mathbb{R}^2$ is not totally disconnected, that is
      $A$ contains a connected component which is not a singleton.
      Then
    \begin{equation}\label{095}
      \left(A+S^1\right)^{\circ}\ne\emptyset.
    \end{equation}

  \end{lemma}
  The proof is presented in Section \ref{g57}.
  It is immediate that for a  set $A$ containing a path (continuous image of $[0,1]$), \eqref{095} holds. However, as the example of Kastner-Kuratowski fan shows, it is possible that a connected set  contains no paths.

  \subsubsection{The Giant: $(A+S^1)^{\circ}$ may be empty for a very big  set $A$}\label{Giant}
We give a simple construction of a $G_\delta$-set $A\subset \mathbb{R}^2$ of full two-dimensional Lebesgue measure so that $A+S^1$ has empty interior.
\vskip.1in
Set $\mathbb{Q}_2:=\mathbb{Q}\times \mathbb{Q}$, and let
\begin{equation}\label{087}
 A:= \left(
\bigcup\limits_{\mathbf{v}\in\mathbb{Q}_2}
S(\mathbf{v},1)
\right)^c,
\end{equation}
where $S(\mathbf{v},1)$ denotes the set $\{x\in \mathbb{R}^2: |x-v|=1\}$.
Then $\mathcal{L}\mathrm{eb}_2
(A^c)=0$, and $$(A+S^1)\cap \mathbb{Q}_2=\emptyset.$$ We call the set defined in \eqref{087} the Giant.
%
 In particular more is true.
 \begin{fact}
   Let $B\subset \mathbb{R}^2$ be arbitrary and $S \subset\mathbb{R}^2$ be an arbitrary set which is symmetric to the origin ($s\in S$ if and only if $-s\in S$). Then
   \begin{equation*}
     \left(\left(B+S\right)^c+S\right)\cap B=\emptyset.
   \end{equation*}
 \end{fact}
 We obtain the Giant by choice of $B:=\mathbb{Q}_2$ and $\mathcal{S}:=S^1$. By the regularity of the Lebesgue measure, we can choose a compact subset $K$ of the Giant with positive Lebesgue measure. That is
 \begin{fact}
   There is a compact $K\subset\mathbb{R}^2$ with $\mathcal{L}\mathrm{eb}_2(K)>0$ such that
    $\left(K+S^1\right)^{\circ}=\emptyset$.
 \end{fact}


We note that the co-dimension of the Giant (defined in \eqref{087}) is one.  If, on the other hand, the co-dimension of an arbitrary set $A$ is less than one, then $A + S^1 = \mathbb{R}^2$:

\begin{fact} Let $S \subset \mathbb{R}^2$ which is symmetric to the origin.
For an arbitrary set $A\subset \mathbb{R}^2$,
if $\mathrm{codim}(A)<\dim_{\rm H}(S) $, then $A+S=\mathbb{R}^2$.
\end{fact}

\begin{proof}
Indeed, if $x\in \mathbb{R}^2$ then either $(x+S ) \cap A =\emptyset$, or $(x+S) \cap A \neq \emptyset$.  The former case contradicts the assumption that $\mathrm{codim}(A)<\dim_{\rm H}(S) $.  The later case, combined with the symmetry of $S$, imlplies that $x\in A+S $.
\end{proof}

\subsection{The case when $\Gamma$ is a  polygon }\label{w99}
In this section we assume that $\Gamma$ is a piecewise linear curve. We call it a polygon.  We construct full measure sets in the plane so that the arithmetic sum with $\Gamma$ has empty interior.
\\

\begin{theorem}\label{g92}
Let $\Gamma$ be an arbitrary polygon in the plane.  Then there exists a full measure set, $A$, so that $$\left(  A + \Gamma\right)^{\circ} = \emptyset.$$
\end{theorem}
The proof is given in Section \ref{p99}.

\subsubsection{The case when $\Gamma$ is a rotated square }
Let  $\Gamma=N_\theta$ denote the angle-$\theta$ rotated copy  (around the origin in anti-clockwise direction) of the square which is the perimeter of $[-1,1]^2$.

 \begin{theorem}\label{g71}
 There exists a set $A$ in the plane of positive measure so that
 for all $\theta\in[0,\pi)$ \textbf{simultaneously} we have
 \begin{equation}\label{g70}
   (A+N_\theta)^{\circ}=\emptyset.
 \end{equation}
 \end{theorem}
The proof is given in Section \ref{p99} and utilizes the existence of Besicovitch sets.  Besicovitch proved (see \cite[Theorem 11.1]{Mat15}) that there exists a compact set $\widetilde{B}\subset \mathbb{R}^2$ such that
$\widetilde{B}$ contains a line  in every direction but $\mathcal{L}^2(\widetilde{B})=0$.
\\

%
%

\section{History and proof for Proposition \ref{T50} }\label{P50}
\subsection{History  }

Proposition \ref{T50} is a strengthening of the
Erd\H{o}s and Oxtoby Theorem \cite{EO54} which is an extension of a classic theorem of Steinhaus \cite{St20, S72} on the interior of difference sets.

\begin{theorem}[Steinhaus]
If $A$ has positive Lebesgue measure, then the difference set $$A-A = \{x-y: x,y \in A\}$$ contains an open neighborhood of the origin.
\end{theorem}
More generally, if $A,B\subset \mathbb{R}$ measurable sets of positive Lebesgue measure then their algebraic sum $A+B$ contains an interval.
This theorem is an easy consequence of the Lebesgue density theorem.\\

Erd\H{o}s and Oxtoby \cite{EO54} provide a significant extension of Steinhaus' result.
\begin{theorem}[Erd\H{o}s, Oxtoby]
Let $A \subset \mathbb{R}$ and $B\subset \mathbb{R}$ each with positive Lebesgue measure.
If $H:\mathbb{R}^{2}\rightarrow \mathbb{R}^{}$ is a $\mathcal{C}^1$ function on an open set $U$, with $m((A\times B) \cap U)>0$ so that the partial derivatives of $H$ are non-vanishing a.e. on $U$, then
 the interior of the set $$\{H(x,y): x \in A, y \in B\}$$ is non-empty.
\end{theorem}
The topological analogue of Steinhaus Theorem was proved by Piccard
\cite{P39}.
\begin{theorem}[Piccard]
Let $A,B\subset \mathbb{R}$ be Baire sets of second category. Then $A+B$ contains an interval.
\end{theorem}

A nice review of the field and  generalizations of the results above are available in \cite{J05}.
\vskip.5in

\subsection{Preliminaries for the Proof of Proposition \ref{T50}}\label{mainlemma}

In Sections \ref{proppart1}, \ref{a17}, \ref{N42Baire} and \ref{N42} we use the notation and Lemmas below.

\subsubsection{Introduction of the function $g_{c,\alpha}$}\label{N65}

   Let $J_1,J_2$ be compact intervals on $\mathbb{R}$
and let
   $H(\alpha,x,y)\in \mathcal{C}^2(\Lambda\times J_1\times J_2)$, where $\Lambda$ is a parameter interval. Moreover, we assume that the $H_x(\alpha,x,y)$ and $H_y(\alpha,x,y)$ are not vanishing on $\Lambda\times J_1\times J_2$. We are also given the points $u_1\in J_1$ and $u_2\in J_2$.
  The purpose of this Section \ref{N65} is to construct by the Implicit Function Theorem a function $g_{c,\alpha}$ (see \eqref{N46}) which sends a neighborhood of $u_1$ to a neighborhood of $u_2$. Moreover, $g_{c,\alpha}$ satisfies
  $$
  H(\alpha,x,g_{c,\alpha}(x))= c,
  $$
   where $\alpha\in \Lambda$ and $c$ is from a neighborhood of $H(\alpha,u_1,u_2)$.

   Without loss of generality we may assume that
   \begin{equation}\label{N76}
     \frac{H_x(\alpha,x,y)}{H_y(\alpha,x,y)}<0,\quad \forall (\alpha,x,y)\in \Lambda\times J_1\times J_2.
   \end{equation}
   For definiteness, we choose $\alpha_0$ as the center of the interval $\Lambda$. Set $c_0:=H(\alpha_0,u_1,u_2)$, and choose a small  $\delta_0>0$ such that
  $$
  \left[\alpha_0-\delta_0,\alpha_0+\delta_0\right]
  \times
  \left[u_1-\delta_0,u_1+\delta_0\right]
  \times
   \left[u_2-\delta_0,u_2+\delta_0\right]\subset \left(\Lambda\times J_1\times J_2\right)^\circ.
  $$
Set
$$
S:= [c_0-\delta_0,c_0+\delta_0]\times
\left[\alpha_0-\delta_0,\alpha_0+\delta_0\right]
  \times
  \left[u_1-\delta_0,u_1+\delta_0\right]
  \times
   \left[u_2-\delta_0,u_2+\delta_0\right].
$$
  For a $(c,\alpha,x,y)\in S $ we define
  $$
F(c,\alpha,x,y):=H(\alpha,x,y)-c.
  $$
  By assumption $F_y(c_0,\alpha_0,u_1,u_2)\ne 0$.
  Then by Implicit Function Theorem, there exists a neighborhood $M\subset S$
  of $(c_0,\alpha_0,u_1,u_2)$ where
  $F_y$ does not vanish.
  To abbreviate notation we write $\mathbf{X}_0:=(c_0,\alpha_0,u_1)$ and
  $\mathbf{X}:=(c,\alpha,x)$.
  Moreover, also from the Implicit Function Theorem, we obtain that there exists a neighborhood
  $N$ of $\mathbf{X}_0$  and a function $G\in \mathcal{C}^2(N)$ so that
    \begin{description}
    \item[(i)] $G(\mathbf{X}_0)=u_2$,
    \item[(ii)] $\left(\mathbf{X},G(\mathbf{X})\right)\in M$,
    \item[(iii)] $F(\mathbf{X},G(\mathbf{X}))=0$ if $\mathbf{X}\in N$,
    \item[(iv)]
    $
 G'(\mathbf{X})=
\left(
\frac{1}{H_y\left(\alpha,x,G(\mathbf{X})\right)},
-\frac{H_\alpha\left(\alpha,x,G(\mathbf{X})\right)}
{H_y\left(\alpha,x,G(\mathbf{X})\right)},
-\frac{H_x\left(\alpha,x,G(\mathbf{X})\right)}
{H_y\left(\alpha,x,G(\mathbf{X})\right)}
\right)
    $.
  \end{description}
  For simplicity we may assume that $N$ is of the form
  \begin{equation*}
    N=
  [c_0-\delta_1,c_0+\delta_1 ]\times
\left[\alpha_0-\delta_1,\alpha_0+\delta_1\right]
  \times
  \left[u_1-\delta_1,u_1+\delta_1\right],
  \end{equation*}
  for a $0<\delta_1<\delta_0$. For a $(c,\alpha)\in [c_0-\delta_1,c_0+\delta_1]\times
\left[\alpha_0-\delta_1,\alpha_0+\delta_1\right]$, we introduce
\begin{equation}\label{N46}
  g_{c,\alpha}(x):=G(c,\alpha,x).
\end{equation}

Then by $(\textbf{iv})$ above, we have
\begin{equation*}\label{N81}
  g'_{c,\alpha}(x)
  =
 -\frac{H_x\left(\alpha,x,G(c,\alpha,x)\right)}
{H_y\left(\alpha,x,G(c,\alpha,x)\right)}.
\end{equation*}
 Recall that by \eqref{N76}, $g'_{c,\alpha}( \cdot )$ is always positive.
  By assumption we can choose an $\eta>0$ such that  for all $(c,\alpha,x)\in N$:
  \begin{equation}\label{N80}
  \eta<|H_x\left(\alpha,x,G(c,\alpha,x)\right)|
 < \frac{1}{\eta},  \mbox{   and   }
  \eta<
|H_y\left(\alpha,x,G(c,\alpha,x)\right)|
 < \frac{1}{\eta}.
  \end{equation}
Then we have
\begin{equation}\label{N79}
  \eta^2 \leq g'_{c,\alpha}(x) \leq \frac{1}{\eta^2}\mbox{ and }
  |g''_{c,\alpha}(x)| \leq \frac{4}{\eta^5}.
\end{equation}
 \begin{lemma}\label{N74}
    For all $ z\in[u_1-\delta_1,u_1+\delta_1]$ we have
    \begin{equation*}\label{N73}
      |g_{c,\alpha}(z)-g_{c_0,\alpha_0}(z)| \leq \frac{2}{\eta^2} \cdot \|(c,\alpha)-(c_0,\alpha_0)\|_{\max}.
    \end{equation*}
  \end{lemma}
  \begin{proof}[Proof of the Lemma]
    \begin{multline*}
 |g_{c,\alpha}(z)-g_{c_0,\alpha_0}(z)|
 =
 |G(c,\alpha,z)-G(c_0,\alpha_0,z)| \\
  \leq |G(c,\alpha,z)-G(c_0,\alpha,z)|
  +
  |G(c_0,\alpha,z)-G(c_0,\alpha_0,z)|\\
  \leq
 |G_c(\widehat{c},\alpha,z)||c-c_0|+
 |G_\alpha(c_0,\widehat{\alpha},z)| \cdot |\alpha-\alpha_0|\\
 \leq\frac{2}{\eta^2}\|(c,\alpha)-(c_0,\alpha_0)\|_{\max},
    \end{multline*}
    where $\widehat{c}$ and $\widehat{\alpha}$ are chosen using the mean value theorem.
  \end{proof}

\begin{lemma}\label{N68}
  Assume that there exists a $\widehat{\varepsilon}\in (0,\delta_1)$ and $K_i\subset J_i$, $i=1,2$ such that
\begin{equation}\label{N70}
  g_{c,\alpha}(K_1)\cap K_2\ne\emptyset\quad  \mbox{if  }\quad
  \|(c,\alpha)-(c_0,\alpha_0)\|_{\max}<\widehat{\varepsilon}.
\end{equation}
Let $J:=(c_0-\widehat{\varepsilon} ,\,\,c_0+\widehat{\varepsilon}\,)$.  Then
\begin{equation}\label{N67}
  J\subset \cap_{\alpha\in(\alpha_0-\widehat{\varepsilon},\,\,\alpha_0+\widehat{\varepsilon})}H(\alpha,K_1,K_2).
\end{equation}
\end{lemma}

\begin{proof}
 Choose an arbitrary
$(c,\alpha)\in J\times (\alpha_0-\widehat{\varepsilon},\alpha_0+\widehat{\varepsilon})$. We claim that
\begin{equation*}\label{N69}
  c\in \left\{H(\alpha,x,y):x\in K_1,y\in K_2\right\}.
\end{equation*}
By assumption, there exists $k_1\in K_1$ and $k_2\in K_2$ such that
$G(c,\alpha,k_1)=k_2$. Then by \textbf{(iii)}
$$
c=H(\alpha,k_1,G(c,\alpha,k_1))=H(\alpha,k_1,k_2).
$$
This completes the proof of the Lemma.
\end{proof}

\subsection{Proof of Proposition  \ref{T50} (i)}\label{proppart1}

\begin{proof}
We use the notation of Section \ref{N65}. Let $u_1$ and $u_2$ be Lebesgue density points of $A$ and $B$ respectively.
 Let $\epsilon_0>0 $ be arbitrarily small.
Choose $0<\widehat{\delta}<\delta_1$ such that
  for $\widetilde{J}_1:=[u_1-\widehat{\delta},u_1+\widehat{\delta}]$ and $\widetilde{A}:=\widetilde{J}_1\cap A$
  we have
  \begin{equation*}\label{N60}
    \frac{|\widetilde{A}|}{|\widetilde{J}_1|}>1-\epsilon_0,
  \end{equation*}
where $|A|:=\mathcal{L}(A)$.
Let $\delta'>0$ small enough so that for $\widetilde{J}_2:=[u_2-\delta',u_2+\delta']$ and $\widetilde{B}:=\widetilde{J}_2\cap B$, we have
  \begin{equation}\label{N62}
    \frac{|\widetilde{B}|}{|\widetilde{J}_2|}>1-\epsilon_0.
  \end{equation}

Moreover, we require that $\widehat{\delta}$ is small enough so that we can choose an $\widetilde{\varepsilon}$ such that
  \begin{equation}\label{N64}
    \|(c,\alpha)-(c_0,\alpha_0)\|_{\max}<\widetilde{\varepsilon}
  \end{equation}
  implies by Lemma \ref{N74} that \eqref{N62} holds with the following choice of $\delta'$:
 letting $\eta>0$ as in \eqref{N80}, we choose
    \begin{equation}\label{N63}
    \delta' = \widehat{\delta} \cdot \frac{1}{\eta^2}+\|g_{c,\alpha}-g_{c_0,\alpha_0}\|_{\max},
  \end{equation}
  where
    $$
  \|g_{c,\alpha}-g_{c_0,\alpha_0}\|_{\max}:=
  \max\left\{|g_{c,\alpha}(u)-g_{c_0,\alpha_0}(u)|:u\in[u_1-\delta_1,u_1+\delta_1]\right\}.
  $$
%
\\

   Further we also require that $\widetilde{\varepsilon}>0$ is small enough that \eqref{N64} implies that
  \begin{equation}\label{N61}
    \|g_{c,\alpha}-g_{c_0,\alpha_0}\|_{\max}<\frac{1}{2}\eta^2\widehat{\delta}.
  \end{equation}

The purpose of \eqref{N61} is to ensure that
\begin{equation}\label{N59}
    u_2 \in \left ( g_{c,\alpha}(\widetilde{J_1})\right)^{\circ}.
\end{equation}
Namely,
 \eqref{N61} implies that
 \begin{equation}\label{z99}
   \left|
   g_{c,\alpha}(u_1)- g_{c_0,\alpha_0}(u_1)
   \right|
   =
    \left|
g_{c,\alpha}(u_1)- u_2
   \right|
   <\frac{1}{2}\eta^2\widehat{\delta}.
 \end{equation}
On the other hand,
\begin{equation}\label{z98}
   \left|
    g_{c,\alpha}(u_1\pm \widehat{\delta})- g_{c,\alpha}(u_1)
   \right|>\widehat{\delta} \cdot \eta^2.
\end{equation}
That is $ g_{c,\alpha}(J_1)$ contains the $\widehat{\delta} \cdot \eta^2$-neighborhood of $g_{c,\alpha}(u_1)$  and $_2$ is contained in this neighborhood, which yields that \eqref{N59} holds.

Now,
    \begin{equation}\label{N58}
    \frac{|g_{c,\alpha}(\widetilde{A})|}{|g_{c,\alpha}(\widetilde{J}_1)|}
    =
1- \frac{|g_{c,\alpha }(   \widetilde{J_1}   \backslash \widetilde{A})|}{|g_{c,\alpha}(   \widetilde{J_1}   )|} \geq 1- \frac{\frac{1}{\eta^2}\epsilon_0 |\widetilde{J_1}|}{  \eta^2 |\widetilde{J_1}|} \geq 1-\frac{\epsilon_0}{\eta^4}.
  \end{equation}

 Next, we obtain a lower bound on $ \frac{|\widetilde{B}\cap g_{c,\alpha}(\widetilde{J}_1)|}{|g_{c,\alpha}(\widetilde{J}_1)|}$ using the assumption that $u_2$ is a density point.
  Fix $(c,\alpha)$ and let $g_{c,\alpha}(\widetilde{J}_1) = (u_2- \delta_2, u_2 + \delta_2')$.  Combining \eqref{N79}, \eqref{N63} and \eqref{N61} by the Mean Value Theorem we obtain that $0<\delta_2, \delta_2' \le \delta'$.
  It follows that
  $$|\widetilde{B}^c \cap [u_2, u_2 + \delta_2')| \le 2\delta_2' \epsilon_0,$$
    $$|\widetilde{B}^c \cap (u_2 - \delta_2,  u_2 ]| \le 2\delta_2 \epsilon_0,$$
    and so
 $$| \widetilde{B}^c \cap g_{c,\alpha}(\widetilde{J_1}) |
 = | \widetilde{B}^c \cap (u_2- \delta_2, u_2 + \delta_2')| \le 2\epsilon_0  |g_{c,\alpha}(\widetilde{J_1}) |  .$$

  Putting this together, we obtain that
  \begin{equation}\label{N57}
    \frac{|\widetilde{B}\cap g_{c,\alpha}(\widetilde{J}_1)|}{|g_{c,\alpha}(\widetilde{J}_1)|}
    >
1- 2\epsilon_0.
  \end{equation}

  Combining \eqref{N58} with \eqref{N57} and
  choosing $(c,\alpha)$ sufficiently close to $(c_0,\alpha_0)$, we get
  \begin{equation*}\label{N56}
    g_{c,\alpha}(\widetilde{A})\cap \widetilde{B}\ne\emptyset.
\end{equation*}
\end{proof}


\subsection{A variant of Newhouse thickness}\label{N43}

In order to prove part (ii) of Proposition \ref{T50}, we need to introduce a modification of the well-known Newhouse thickness (see
\cite{PT00}). Namely, we have to tackle the problem that
 the Newhouse thickness of a Cantor set can drop significantly if we take a smooth image of a Cantor set.
\begin{definition}\label{N100}
  Let $K\subset\mathbb{R}$ be a Cantor set. The gaps of $K$ are the connected components of $K^c$.
  Let us denote the collection of gaps of $K$ by $\mathcal{G}:=\mathcal{G}_K$. We write $\ell (G)$ and $r(G)$ for the left and right endpoints of the gap $G$ correspondingly.
 Let $u\in K$ be an end point of a gap $G$. Without loss of generality we may assume that $u=r(G)$.
 For an $\varepsilon \geq 0$, we define the
  $\varepsilon$-bridge $B_\varepsilon(u)$  as follows:
  \begin{equation}\label{N95}
    B_\varepsilon(u):=\left(u,\ell (\widetilde{G})\right),
  \end{equation}
  where $|\widetilde{G}|  \geq (1-\varepsilon)|G|$ and, whenever $\widehat{G}\in\mathcal{G}$ with
  $\widehat{G}\subset (u,\ell (\widetilde{G}))$, then $|\widehat{G}|<(1-\varepsilon)|G|$. Now we can define the $\varepsilon$-thickness of $K$ at $u$ by
  \begin{equation*}
  \tau_\varepsilon(K,u):=\frac{|B_\varepsilon(u)|}{|G|}
  \mbox{ and }
  \tau_\varepsilon(K):=\inf\limits_{u\in U}\tau_\varepsilon(K,u),
  \end{equation*}
  where $U$ is the set of the endpoints of the gaps. Note that the case of $\varepsilon=0$ is the usual Newhouse thickness of a Cantor set (see \cite[p.61]{PT00}).
  \end{definition}
  It is straightforward that
  \begin{equation}\label{N99}
    \mbox{If }
    0 \leq \varepsilon_1 <\varepsilon_2 \mbox{ then }
    \tau_{\varepsilon_1}(K)   \geq  \tau_{\varepsilon_2}(K).
  \end{equation}
  This is so because for every gap endpoint $u\in U$ we have
  $B_{\varepsilon_1}(u) \geq B_{\varepsilon_2}(u)$.
  \begin{lemma}\label{N45}
   For every Cantor set $C\subset \mathbb{R}$ and $\epsilon \in (0,1)$ we have
   $$(1-\epsilon)^2  \le \frac{\tau_{\epsilon}(C)}{\tau(C)} \le 1.$$
   In particular,
  $$\lim_{\epsilon\rightarrow 0}\tau_{\epsilon}(C) = \tau(C).$$

  \end{lemma}
  \begin{proof}
  Let $C\subset \mathbb{R}$ be a Cantor set.
In order to get a contradiction,  we assume that there exists an $\epsilon \in (0,1) $ so that
\begin{equation}\label{contra}\tau_{\epsilon}(C) < (1-\epsilon)^2 \tau(C).\end{equation}
Let $U$ be the set of the endpoints of the gaps of $C$.
By definition, we can find $u\in U$ so that
$$\tau_{\epsilon}(C) \le \tau_{\epsilon}(C,u) < \frac{1}{1-\epsilon}\cdot  \tau_{\epsilon}(C) .$$
Without loss of generality, we may assume that $u$ is the right-end-point of a gap $G$.
Define $\widehat{G}$ to be the first gap of $C$ to the right of $u$ so that
\begin{equation*}(1-\epsilon) | G| \le |\widehat{G} |. \end{equation*}
Define $\widetilde{G}$ to be the first gap of $C$ to the right of $u$ so that
\begin{equation}\label{T98}| G| \le |\widetilde{G} |. \end{equation}
Set $\widetilde{u}: = l(\widetilde{G}),$ and $\widehat{u}: = l(\widehat{G}).$  Clearly,
\begin{equation}\label{T97}B_{\epsilon}(u) = [u, \widetilde{u}] ,
  \mbox{ and }
B_0(u) = [u, \widehat{u}] , \end{equation}
It follows by assumption \eqref{contra} above that
\begin{equation}\label{T96}
B_{\epsilon}(u) \subset B_0(u)
\mbox{ and }
B_{\epsilon}(u) \neq B_0(u).
\end{equation}
As a consequence, we have
\begin{equation}\label{T95} (1-\epsilon) |G| \le |\widehat{G}| < |G| \le |\widetilde{G}|.\end{equation}
We claim that
\begin{equation}\label{T94}
B_0(\widehat{u}) = [u, \widehat{u}].
\end{equation}
On the one hand, $B_0(\widehat{u}) \subset [u, \widehat{u}]$ follows from the second inequality in \eqref{T95}.
On the other hand, if $B_0(\widehat{u}) \neq [u, \widehat{u}]$, then $\exists$ a gap $G' \subset (u, \widehat{u})$ so that
$B_0(\widehat{u})  = [r(G') , \widehat{u}]$ and $|G'| \geq |\widehat{G}| \geq (1-\epsilon) |G|.$  The existence of such a $G'$ contradicts the definition of $\widehat{G}.$

Now we have
\begin{align*}
\tau_0(C) \le \tau_0(C,u)
= \frac{|B_0(\widehat{u})  |}{    |\widehat{G}|}
\le \frac{|B_\epsilon(u)|}{(1-\epsilon) |G|}
\le  \frac{1}{(1-\epsilon)^2}\cdot \tau_{\epsilon}(C) .
\end{align*}
It follows
that $$(1-\epsilon)^2\cdot \tau_{0}(C) < \tau_{\epsilon}(C),$$ which contradicts our assumption in \eqref{contra}.
  \end{proof}

  \begin{lemma}\label{N98}Let $K\subset \mathbb{R}$ be a Cantor set, and let $I$ be an open interval such that
  $I\cap K$ is a closed nonempty set.
    Let $g\in \mathcal{C}^1(\mathbb{R})$ satisfying
    $\min\limits_{x\in I}g'(x)>0$. Moreover,  we have
    \begin{equation}\label{N97}
      1-\varepsilon \leq
      \frac{g'(z_1)}{g'(z_2)}
       \leq  1+\varepsilon, \quad \forall z_1,z_2\in I
    \end{equation}
  Then
  \begin{equation}\label{N96}
    \tau_{\varepsilon^2}(g(I\cap K)) \geq \tau_{\varepsilon}(K)(1-\varepsilon).
  \end{equation}
  \end{lemma}
  \begin{proof}
    Clearly, the image of $K\cap I$ by $g$ is a Cantor set and the gaps of the image $g(I\cap K)$ are the images of the gaps of $K\cap I$. Now we use the notation of Definition \ref{N100} in particular the one in \eqref{N95}.
    We claim that
 \begin{equation}\label{N93}
   B_{\varepsilon^2}(g(u))\supset
   g(B_{\varepsilon}(u)).
 \end{equation}
   To see this we fix an arbitrary gap $G'$ contained in $B_\varepsilon(u)$.
    Using the mean value theorem we can find $z,z',\widetilde{z}$ such that
    \begin{equation}\label{N94}
   \frac{|g(G)|}{|G|}=g'(z),\
   \frac{|g(G')|}{|G'|} =g'(z'), \
    \frac{|g(\widetilde{G})|}{|\widetilde{G}|}=g'(\widetilde{z})\mbox{ and }
    \frac{|g(B_\varepsilon(u))|}{|B_\varepsilon(u)|}=g'(z_\varepsilon).
    \end{equation}
Observe that
$$
\frac{|g(G')|}{|g(G)|}
=
\frac{|G'|}{|G|}
 \cdot
\frac{|g(z')|}{|g(z)|}
<
(1-\varepsilon)(1+\varepsilon)=1-\varepsilon^2.
$$
This verifies \eqref{N93} by the definition of the $\varepsilon^2$-bridge.
    Using \eqref{N93} we can write
    $$
\tau_{\varepsilon^2}(g(K\cap I),g(u))=
\frac{|B_{\varepsilon^2}(g(u))|}{|g(G)|}   \geq
\frac{|g(B_\varepsilon(u))|}{|g(G)|}=\frac{|B_\varepsilon(u)|}{|G|} \cdot
\frac{g'(z_\varepsilon)}{g'(z)}.
    $$
    Using \eqref{N97} and  taking the infimum over the gap  endpoints of $K\cap I$
     completes the proof of the Lemma.
  \end{proof}


\subsection{Proof of Proposition \ref{T50} (ii)}\label{a17}

\begin{proof}
In order to emphasize that the sets $A$ and $B$ are Cantor sets in this Subsection we are going to call them
$K_1$ and $K_2$ respectively.  The smallest intervals containing them are $J_1,J_2$ respectively. By assumption we know that $\tau(K_1) \cdot \tau(K_2)>1$. Hence by Lemma \ref{N45}  there exists an $\varepsilon>0$
\begin{equation}\label{N91}
      \tau(K_2) \cdot \tau_\varepsilon(K_1)>1.
    \end{equation}
    Throughout this Section we use the notation of Section \ref{N61}.
   Let $G_1, G_2$ be bounded gaps of $K_1,K_2$ respectively and let $u_i:=r(G_i)$, $i=1,2$.
   Choose a $\delta>0$ such that $r(G_1)+\delta\in J_1$ and
   \begin{equation}\label{a16}
     \delta \leq \min\left\{\delta_1,\frac{\varepsilon \cdot \eta^7}{8}\right\}.
   \end{equation}

 Let $\widetilde{K}_1:=[u_1,u_1+\delta]\cap K_1$.
 Given a strictly increasing  mapping $g_0\in \mathcal{C}^1[u_1,u_1+\delta]$ such that $g_0(u_1)=u_2$,
    let $\widetilde{K}_2:=g_0([u_1,u_1+\delta])\cap K_2$.
     Now we choose $z_1,z_2\in \widetilde{K}_1$ and $v_1,v_2,v_3\in g_{0}^{-1}(\widetilde{K}_2)$ in such a way that
    \begin{equation}\label{N86}
      v_1<z_1<v_2<z_2<v_3.
    \end{equation}
    This is possible because $u_i$ are accumulation points of the Cantor sets $\widetilde{K}_1$ and $\widetilde{K}_2$.
    Now we define $0<\widetilde{\varepsilon}<\varepsilon$ such that
    \begin{equation}\label{N72}
      \widetilde{\varepsilon}:=\frac{1}{3}\min\left\{|g_0(z_i),g_0(v_j)|:i=1,2,\ j=1,2,3\right\}
    \end{equation}
    Now we introduce $g\in \mathcal{C}^1[u_1,u_1+\delta]$ such that $g$ satisfies \eqref{N97} with $I=[u_1,u_1+\delta]$
    where $\| \cdot \|$ denotes the supnorm on the interval $[u_1,u_1+\delta]$.

     \begin{lemma}\label{N92}
    Let $g\in \mathcal{C}^1[u_1,u_1+\delta]$ be a function satisfying
    \eqref{N97} with $I=[u_1,u_1+\delta]$ and
  \begin{equation}\label{N85}
      \|g-g_0\|<\widetilde{\varepsilon}.
    \end{equation}
   Then
    \begin{equation}\label{N90}
      g(\widetilde{K}_1)\cap \widetilde{K}_2\ne\emptyset.
    \end{equation}
  \end{lemma}

  \begin{proof}
    Without loss of generality we may assume that even the following inequality holds:
    \begin{equation}\label{N89}
      \tau(\widetilde{K}_2) \cdot \tau_\varepsilon(\widetilde{K}_1)(1-\varepsilon)>1.
    \end{equation}
    since changing to a smaller $\varepsilon$ increases  the left hand side.
    Then by \eqref{N96} we obtain that
    $\tau(\widetilde{K}_2) \cdot \tau_{\varepsilon^2}(g( \widetilde{K}_1))>1.$ Using \eqref{N99} we obtain that
    \begin{equation}\label{N88}
      \tau(\widetilde{K}_2) \cdot \tau(g( \widetilde{K}_1))>1.
    \end{equation}
    Now we use Newhouse gap Lemma \cite[p. 63]{PT00}. This implies that
    \begin{equation}\label{N87}
      K_2\cap g(I\cap K_1)\ne\emptyset
    \end{equation}
    since the Cantor sets $K_2$ and  $g(I\cap K_1)$ cannot possibly in each others gap. This follows from the choice of $\widetilde{\varepsilon}$ and \eqref{N85}.
  \end{proof}


In the next Fact we verify that the condition of Lemma \ref{N98} holds.

\begin{fact}\label{N78}
  Let $(c,\alpha)\in (c_0-\delta_1,c_0+\delta_1)\times(\alpha_0-\delta_1,\alpha_0+\delta_1)$.
 Then whenever $\zeta_1,\zeta_2\in [u_1-\delta,u_1+\delta]$
  we have
  \begin{equation}\label{N77}
    \frac{g'_{c,\alpha}(\zeta_1)}{g'_{c,\alpha}(\zeta_2)}
    \in\left(1-\varepsilon,1+\varepsilon\right),
  \end{equation}
  recall that $\delta$ was introduced in \eqref{a16}.
\end{fact}
  \begin{proof}[Proof of the Fact]
  Using mean value theorem there exists a $\zeta_3\in(\zeta_1,\zeta_2)$ such that
  $$
  \left|\frac{g'_{c,\alpha}(\zeta_1)}{g'_{c,\alpha}(\zeta_2)} -1\right|=
      \left| \frac{g'_{c,\alpha}(\zeta_1)-g'_{c,\alpha}(\zeta_2)}{g'_{c,\alpha}(\zeta_2)}\right|
      \leq
      \frac{1}{\eta^2}
      |g_{c,\alpha}''(\zeta_3)| \cdot |\zeta_1-\zeta_2|
       \leq \frac{4}{\eta^7}|\zeta_1-\zeta_2| \leq \varepsilon,
  $$
  where  we used first \eqref{N79} and then \eqref{a16}.
  \end{proof}
We now verify the condition of Lemma \ref{N92}.

Fix an arbitrary $(c,\alpha)$ satisfying
\begin{equation}\label{N71}
 \|(c,\alpha)-(c_0,\alpha_0)\|_{\max}< \frac{\widetilde{\varepsilon} \cdot \eta^2}{2}.
\end{equation}
Then the conditions of Lemma \ref{N74} hold, so by Lemma \ref{N74} we obtain that
\begin{equation}\label{a15}
  \|g_{c,\alpha}-g_{c_0,_0}\|<\widetilde{\varepsilon}.
\end{equation}
Thus,  by Lemma \ref{N92}, we obtain that
\begin{equation}\label{a14}
  g_{c,\alpha}(\widetilde{K}_1)\cap \widetilde{K}_2\ne\emptyset\quad  \mbox{if  }\quad
  \|(c,\alpha)-(c_0,\alpha_0)\|_{\max}<\frac{\widetilde{\varepsilon} \cdot \eta^2}{2}.
\end{equation}
Next we apply Lemma \ref{N68} with $\widehat{\varepsilon}=\frac{\widetilde{\varepsilon} \cdot \eta^2}{2}$ and from
\eqref{N67} we obtain that
the assertion of Proposition \ref{T50} (ii) holds.


\subsection{Proof Proposition \ref{T50} (iii)}\label{N42Baire}

We write $\mathbb{K}$ and ($\mathbb{S}$) for the collection of sets of Baire first (second) category on the line respectively.

Recall that a  set  is of second category if it is not a set of first category. Moreover, a set is of first category if it can be represented as a countable union of nowhere dense sets. A set is nowhere dense if it is not dense in any balls.

Recall also that the topological analogues of Lebesgue measurable sets are the so-called Baire sets: We say that $A$ is a Baire set
if there exists an open set $E$ and a set of first category $M$ such that $A=E\triangle M$.

The steps of the following proof are just a combination of the ones from \cite[Theorem 4.1,Remark 4.2]{J05} but  an immediate application of \cite[Theorem 4.1,Remark 4.2]{J05}  yield only that $H(\alpha,A,B)^{\circ}\ne\emptyset$. However, we need more. Namely, that for a suitable parameter interval $I$, we have
$\bigcap_{\alpha\in I}H(\alpha,A,B)\ne\emptyset$, and this is what we prove below.

First we state a well-known fact:

\begin{fact}\label{a98}
Let $A$ be a subset of a complete separable metric space $(X,\varrho)$
\begin{equation}\label{a97}
  R_A:=\left\{x\in X:
  \forall \delta>0, B(x,\delta)\cap A \mbox{ is set of  second category }
  \right\}.
\end{equation}
Then $R_A$ is closed and $A\setminus R_A^{\circ}$ is a set of first category.
\end{fact}

  We can present  $B=E_2 \setdiff M_2$, where $E_2$ is an open set and $M_2\in\mathbb{K}$.
  By the assumption of part (iii) of Proposition \ref{T50}
   and Fact \ref{a98}, we can choose
  \begin{equation}\label{a95}
    u_1\in A\cap R_{A}^{\circ} \mbox{ and }
u_2\in B\cap E_2\cap R_{B}^{\circ}
  \end{equation}
  We use the notation
of Section \ref{mainlemma}. In particular we  choose $\alpha_0,c_0$ as in Section \ref{mainlemma}.
Let
\begin{equation}\label{a88}
  Z:=\left\{
  (c,\alpha)\in(c_0+\delta,c_0-\delta)\times(\alpha_0-\delta,\alpha_0+\delta)
  :
 A\cap  g_{c,\alpha}^{-1}(B)\in\mathbb{S}
  \right\},
\end{equation}
The rest of the proof is organized as follows:
First we prove that
\begin{equation}\label{a85}
 (c_0,\alpha_0)\in Z.
\end{equation}
(This will be easy.)
Then we verify  the much more difficult Lemma:
\begin{lemma}\label{a86}
  $Z$ is open.
\end{lemma}

\begin{proof}[Proof of Proposition \ref{T50} (iii), assuming \eqref{a85} and Lemma \ref{a86}]
 It follows from \eqref{a85} and and Lemma \ref{a86} that the condition
\eqref{N70} of Lemma \ref{N68} holds.
Then we apply Lemma \ref{N68} which completes the proof of Proposition \ref{T50} (iii).

\end{proof}

Below we use several times   that by the first part of \eqref{N79}, for every $(c,\alpha)$ we have
$$
g_{c,\alpha}^{-1} (H)\in\mathbb{K} \mbox{ whenever } H\in\mathbb{K}.
$$

\textit{Proof of} \eqref{a85}:
\begin{align*}
  A\cap g_{c_0,\alpha_0}^{-1} (B)
= &  A\cap \left(g_{c_0,\alpha_0}^{-1} (E_2)\setdiff g_{c_0,\alpha_0}^{-1} (M_2)\right)\\
  \supset &
\left(A \cap g_{c_0,\alpha_0}^{-1} (E_2)\right)\setminus g_{c_0,\alpha_0}^{-1} (M_2)\in \mathbb{S}.
\end{align*}
The last inclusion follows from the definition of $R_A$ and the fact that $g_{c_0,\alpha_0}^{-1} (E_2)$ is a neighborhood of $g_{c_0,\alpha_0}^{-1} (u_2)=u_1\in A\cap R_{A}^{\circ}$.
\end{proof}
So, in the rest of this Subsection we prove Lemma \ref{a86}.

\begin{proof}[Proof of Lemma \ref{a86}]
We define
$$
K:=A\cap g_{c_0,\alpha_0}^{-1} (B),\mbox{ and }
K':=K\cap \left( g_{c_0,\alpha_0}^{-1} (E_2)\cap R_{A}^{\circ}\right).
$$

It follows from \eqref{a85} that $K\in\mathbb{S}$. Clearly,
\begin{equation}\label{a84}
  K\setminus K'\subset
  g_{c_0,\alpha_0}^{-1}(M_2)\cup \left(A\setminus R_{A}^{\circ}\right) \in\mathbb{K}.
\end{equation}
That is $K'\in\mathbb{S}$.
Choose an arbitrary $y\in K'\cap R_{K'}$. That is
\begin{equation}\label{a83}
y\in K,\ g_{c_0,\alpha_0}(y)\in E_2, \mbox{ and }
y\in R_{A}^{\circ}.
\end{equation}
In particular $G(c_0,\alpha_0,y)\in E_2 $. (Recall $G$ was defined in \eqref{N46}.)
 Hence, we can choose a neighborhood  $V$ of $(c_0,\alpha_0)$ and a neighborhood $W$ of $y$ such that
\begin{equation}\label{a82}
  \forall (c,v)\in V,\mbox{ we have }
  W\subset g_{c,\alpha}^{-1} (E_2)\cap  R_{A}^{\circ}.
\end{equation}
Now we prove
\begin{fact}\label{a13}
  \begin{equation}\label{a81}
  W\cap A\in\mathbb{S}.
\end{equation}
\end{fact}

\begin{proof}[Proof of Fact \ref{a13}]
To get contradiction assume that $W\cap A\in \mathbb{K}$. Then
$W\setminus A=W\setminus (W\cap A)$ is a Baire set since $W$ is open and we assumed
$W\cap A\in \mathbb{K}$. We know that a set is Baire if and only if  it can be presented as the union of a  $G_\delta$-set a set and a set of first category. If $W\cap A\in \mathbb{K}$ then
$W\setminus A\in\mathbb{S}$ therefore there is a nonepmty $G_\delta$ set
\begin{equation}\label{a80}
  G\subset W\setminus A\subset R_{A}^{\circ}\setminus A.
\end{equation}
As a $G_\delta$ set in itself, $G$ is also a Baire set. That is we can find an open set $U$
 and $F\in\mathbb{K}$ such that
 \begin{equation}\label{a79}
    G=U\setdiff F.
 \end{equation}
 We obtain from \eqref{a80} and \eqref{a79} that $U\cap R_{A}^{\circ}\ne \emptyset$.
But $U\setminus F\subset G\subset A^c$. Which implies that $U\subset A^c\cup F$.
We get from this that $U\cap A\subset F$. That is $U\cap A\in\mathbb{K}$. In particular
\begin{equation}\label{a78}
  (R_{A}^{\circ}\cap U)\cap A\in\mathbb{K}.
\end{equation}
On the other hand, using that $U\cap R_{A}^{\circ}\ne \emptyset$ we can select
$x\in U\cap R_{A}^{\circ}  $. Then by the fact that $U$ is open  there exists a $\delta_x>0$ such that
$B(x,\delta_x)\subset U\cap R_{A}^{\circ}$. But by the definition of $R_A$ we obtain
$B(x,\delta_x)\cap A\in\mathbb{S}$. That is
\begin{equation}\label{a77}
(R_{A}^{\circ}\cap U)\cap A\in\mathbb{S}.
\end{equation}
This contradicts with \eqref{a78}. This completes the proof of Fact \ref{a13}.
\end{proof}

Now we fix an arbitrary $(c,\alpha)\in V$. We prove that
\begin{equation}\label{a76}
  W\cap A\cap g_{c,\alpha}^{-1} (B)=
  \underbrace{\left(W\cap A\right)}_{\in\mathbb{S}}\cap \underbrace{(W\cap g_{c,\alpha}^{-1} (B))}_{\mbox{a residual set in $W$}}\in\mathbb{S}.
\end{equation}
Namely, $W\cap A\in\mathbb{S}$ follows from \eqref{a81}. The fact that
$W\cap g_{c,\alpha}^{-1} (B)$ is a residual set in $W$ since
$$
W\setminus (W\cap g_{c,\alpha}^{-1} (B))=
W\setminus (g_{c,\alpha}^{-1} (E_2)\setdiff g_{c,\alpha}^{-1} (M_2))
\subset g_{c,\alpha}^{-1} (M_2)\in\mathbb{K},
$$
where the one but last inclusion follows from \eqref{a82}. This proves \eqref{a76} which completes the proof of Lemma \ref{a86}. Which in turn completes the proof of part (iii) of
Proposition \ref{T50}.

\end{proof}
\subsection{Proof Proposition \ref{T50} (iv)}\label{N42}
\begin{proof}

We  use  again the notation of Section \ref{N65}.
Let $u_1\in \mathrm{int}(J_1\cap A)$ and let $u_2\in B$ be a condensation point
(see \cite[Exercise 27 in Section 2]{R76}) of $B$ that every neighbourhood of $u_2$ contains uncountably many elements of $B$.
Similarly to the proof of part (i), using Lemma \ref{N74} we can choose
 an $\widetilde{\varepsilon}>0$ such that
 $$
 \|(c,\alpha)-(c_0,\alpha_0)\|_{\max}<\widetilde{\varepsilon}
\Longrightarrow
 \|g_{c,\alpha}-g_{c_0,\alpha_0}\|_{\max}<
 \frac{1}{2}\eta^2\delta_1.
$$
Then by the first part of
\eqref{N79} and the Mean Value Theorem,
we obtain that for $\widetilde{J}_1:=[u_1-\delta_1,u_1+\delta_1]$
\begin{equation}\label{a099}
\|(c,\alpha)-(c_0,\alpha_0)\|_{\max}<\widetilde{\varepsilon}
\Longrightarrow  u_2\in \left(g_{c,\alpha}(\widetilde{J}_1)\right)^{\circ}.
\end{equation}
That is $\left(g_{c,\alpha}(\widetilde{J}_1)\right)^{\circ}$ is a neighborhood of $u_2$ hence it contains uncountably many elements of $B$:
$$
\#\left(B\cap\left(g_{c,\alpha}(\widetilde{J}_1)\right)^{\circ}\right)=\aleph.
$$
Hence, $g_{c,\alpha}^{-1} (B)\cap \widetilde{J}_1=\aleph$. On the other hand,
$\#(\widetilde{J}_1\setminus A)=\aleph_0$. So we get
$$
\|(c,\alpha)-(c_0,\alpha_0)\|_{\max}<\widetilde{\varepsilon}
\Longrightarrow g_{c,\alpha}^{-1} (B)\cap  A\ne\emptyset.
$$
By Lemma \ref{N68} this completes the proof.
\end{proof}

\medskip

\section{Proof of Theorem \ref{main}}\label{a02}
\begin{proof}
Let  $\widetilde{\Gamma}\subset \Gamma$ such that $\widetilde{\Gamma}$ is a
  $\mathcal{C}^2$ curve with parametrization $x\mapsto (x,\gamma(x))$ with
$0 \leq x \leq a$
and $\gamma''(x) \neq 0$.

  Without loss of generality we may assume that
  \begin{equation}\label{N55}
    A\times B\subset \left(0,\frac{a}{4}\right]^2.
  \end{equation}
Recall
  \begin{equation}\label{N49}
    P(A,B)=
\left\{
(x,z(x,y)):x\in A\mbox{ and } y\in B
\right\},
  \end{equation} where $z$ was defined in \eqref{a01}
  \\

Fix $\alpha\in (a/4,a)$, and observe that for $\widetilde{a}\in A$, $\widetilde{b}\in B$, we have $(\widetilde{a},z(\widetilde{a},\widetilde{b})) \in P(A,B)$ and $(\alpha-\widetilde{a}, \gamma(\alpha-\widetilde{a})) \in \widetilde{\Gamma}$, and so
\begin{equation}\label{key} (\alpha, z(\widetilde{a},\widetilde{b})+\gamma(\alpha-\widetilde{a})) \in P(A,B) + \widetilde{\Gamma}. \end{equation}

Set
\begin{equation}\label{N50}
 H(\alpha,x,y):=
z(x,y)+\gamma(\alpha-x).
\end{equation}

We verify that $H$ satisfies the hypotheses of Proposition \ref{T50} in a sufficiently small open subset of
$
\left(\frac{a}{4},a\right)
\times
\left(0,\frac{a}{4}\right]^2.
$
Fix an $(u_1, u_2) \in \left(0,\frac{a}{4}\right]^2 $
as in the proof of Proposition \ref{T50}.
Using the assumption that the curvature of $\gamma$ is non-vanishing we can find an $\alpha_0\in \left(\frac{a}{4},a\right)$ such that
$$
  z_x(u_1,u_2)\ne \gamma'(\alpha_0-u_1).
$$

This and \eqref{a01} together imply that  $H(\alpha,x,y)\in \mathcal{C}^2(\Lambda  \times  J_1\times J_2)$ with non-vanishing partial derivatives in $x$ and $y$, where
$\Lambda  \times  J_1\times J_2$ is
  a sufficiently small neighborhood of $(\alpha_0,u_1,u_2)$ satisfying $\Lambda  \times  J_1\times J_2\subset
  \left(\frac{a}{4},a\right)
\times
\left(0,\frac{a}{4}\right]^2
$.
  \\

Observe that we may choose $u_1,u_2$ for each of the four parts of Theorem \ref{main} and Proposition \ref{T50}
such that if we replace $A,B$ with
$$
\widehat{A}:=J_1\cap A,\quad
\widehat{B}:=J_2\cap B
$$
respectively,
 then $\widehat{A}$, $\widehat{B}$ preserve the same property
that $A,B$ were characterized with in the assumptions of (i)-(iv) of the Theorem \ref{main} and Proposition \ref{T50}.
\\

Hence, by Proposition \ref{T50} applied with
$\widehat{A}, \widehat{B}$
instead of $A,B$, we obtain that
there exists an open interval $I \subset \Lambda$
such that
$$
 \left( \bigcap_{\alpha\in I}H(\alpha, \widehat{A}, \widehat{B})\right)^0\neq \emptyset.
$$

Combining this result with the observation in \eqref{key} and the definition of $H$ given in \eqref{N50}, we conclude that
  $\left(P(A,B)+\widetilde{\Gamma}\right)^{\circ}\neq\emptyset$.
  \\

\end{proof}

\section{Proof of Theorem \ref{annulus} and Lemma \ref{g99}}

\subsection{Proof of Theorem \ref{annulus}}\label{annulusproof}

\begin{proof}[Proof of Theorem \ref{annulus}]
We consider the case when $0$ is a density of both $A$ and $B$, and we show that $\left(A\times B \right)+S^1$ contains a neighborhood of $S^1$. An appropriate shift of $A\times B$ allows us to reduce to this case.
\\

Let $\epsilon>0$ and apply the Lebesgue density theorem to choose $\delta_0>0$ so that if $\delta_0\geq \delta>0$, then
\begin{equation}\label{Adensity}\mathcal{L}^1(A\cap (-\delta,0))> \delta \cdot (1-\epsilon) \mbox{ \,\,\, and \,\,\,  } \mathcal{L}^1(A\cap (0,\delta))> \delta \cdot(1-\epsilon).\end{equation}
Also, choose $\delta'_0>0$ so that  if $\delta'_0\geq \delta'>0$, then
\begin{equation}\label{Bdensity}\mathcal{L}^1(B\cap (-\delta',0))> \delta' \cdot (1-\epsilon) \mbox{ \,\,\, and \,\,\,  } \mathcal{L}^1(B\cap (0,\delta'))> \delta' \cdot(1-\epsilon).\end{equation}
\\

Fix $\delta>0$ so that \eqref{Adensity} holds and \eqref{Bdensity} holds with the choice $\delta' = 1-\gamma(\delta)$, where
$\gamma(x) = \sqrt{1-x^2}$.  Observe by Taylor's theorem that $\delta'\sim \delta^2.$
\\

For $\alpha \in [-\delta,\delta]$, consider $g_{c,\alpha}(x) = c-\gamma (\alpha-x)$ and $h_{c,\alpha}(x) = c+\gamma (\alpha-x)$.
\begin{lemma}\label{K1}
For each $\alpha \in [-\delta,\delta]$, there exists a neighborhood of $\gamma(\alpha)$, call it $N(\gamma(\alpha))$, so that if $c\in N(\gamma(\alpha))$, then
\begin{equation}\label{z97}
 g_{c,\alpha}(A) \cap B\neq \emptyset.
\end{equation} Further, the length of $N(\gamma(\alpha))$ is bounded below away from zero uniformly over $\alpha \in [0,\delta]$.   \end{lemma}
The lemma also holds when $g_{c,\alpha}$ is replaced by $h_{c,\alpha}$.
\\

\begin{corollary}\label{NSpoles}
For $\alpha \in [-\delta,\delta]$, it is a simple consequence of the lemma that the set
$$\left\{(\alpha, c) : c\in N(\gamma(\alpha)) \right\}\subset \left(A\times B \right)+S^1.$$
Moreover, replacing $g_{c,\alpha}$ by $h_{c,\alpha}$ in Lemma \eqref{K1} yields that $\left\{(\alpha, c) : c\in N(-\gamma(\alpha)) \right\}\subset \left(A\times B \right)+S^1$.
\end{corollary}
\vskip.125in

To prove the Corollary, we observe that for each $|\alpha|<1$, for each $a\in A$, and for each $b\in B$,
$$\left(\alpha, b\pm \gamma(\alpha-a)\right) \,\, \in \,\, \left( \left(A\times B\right) +S^1 \right) \cap \ell _{\alpha},$$
where $\ell _\alpha:=\left\{(x,y):x=\alpha\right\}$.


We now prove Lemma \ref{K1}.   We consider the case when $\alpha \in [0,\delta]$ (the case when $\alpha \in [-\delta,0 ]$ follows by a similar argument).  The plan is to get a lower bound on
\begin{equation}\label{m4}
\frac{\mathcal{L}^1 \left( J_2^{\alpha}   \cap  g_{c,\alpha}\left(A \cap J_1^{\alpha}\right) \right)}{\mathcal{L}^1\left(  J_2^{\alpha} \right)   },\end{equation}
where $J_1^{\alpha} \subset [-\delta, \delta]$ and $J_2^{\alpha}\subset [-\delta', \delta']$ are small intervals to be decided on below.
\\

We already have by \eqref{Bdensity} that, for any interval $J_2^{\alpha}\subset [-\delta', \delta']$ which contains $0$,
\begin{equation}\label{m5}
\frac{ \mathcal{L}^1\left( J_2^{\alpha} \cap B\right)}{  \mathcal{L}^1\left(  J_2^{\alpha} \right)   } > 1- \epsilon.
\end{equation}

The conclusion of the Lemma will follow provided that the sum of the lower bound on \eqref{m4} and the lower bound on \eqref{m5} is greater than 1.
\\

We consider the following cases separately: \\
case 1a: $\alpha \in [0,\delta]$ and $c=\gamma(\alpha)$; \\
case 1b: $\alpha \in [0,\delta]$ and $c=\gamma(\alpha)+e$ for each $|e| <\frac{\delta^2}{4}$; and\\
case 2: $\delta \le  \alpha <1$ and $c=\gamma(\alpha)+e$ for each $|e| <\frac{\delta^2}{4}$.
\\

\textbf{case 1a:  $\alpha \in [0,\delta]$, $c=\gamma(\alpha)$:}\\
Let $\alpha \in [0,\delta]$, set $c=\gamma(\alpha)$, and set
 $$J_1^{\alpha} = [-\delta +\alpha, \alpha] \subset [-\delta, \delta].$$
Now
$$g_{\gamma(\alpha),\alpha}(x) = \gamma(\alpha) - \gamma(\alpha-x)$$
 is strictly decreasing on $J_1^{\alpha}$.
 Note that $g_{\gamma(\alpha),\alpha}(0)=0$ and so the graph of $g_{\gamma(\alpha),\alpha}$ passes through zero.
 \\

Set
$$J_2^{\alpha}\,\,=\,\, g_{\gamma(\alpha),\alpha}\left(J_1^{\alpha}\right)\,\,= \,\, [\gamma(\alpha)-1, \gamma(\alpha)- \gamma(\delta)]
\,\, \subset \,\, [ \gamma(\delta)-1, 1-\gamma(\delta)].$$
\vskip.125in

Since the length of $J_1^{\alpha} $ is $\delta$ and $J_1^{\alpha}  \subset [-\delta, \delta]$,  it follows by \eqref{Adensity} that $\mathcal{L}^1\left(A^c\cap J_1^{\alpha}\right) < \delta\cdot \epsilon$.  Also, $\left| g_{\gamma(\alpha),\alpha}'(x)\right|< 3\delta $ for  $x\in [-\delta, \delta]$.  Thus
\begin{equation}\label{m1}\mathcal{L}^1\left( g_{\gamma(\alpha),\alpha}\left(A^c\cap J_1^{\alpha}\right)\right) = \left| \int_{A^c\cap J_1^{\alpha}} g_{\gamma(\alpha),\alpha}'(x) dx  \right|
\, \le \, 3\epsilon \delta^2.\end{equation}
Next, we use Taylor's theorem to obtain a uniform lower bound on the length of $J_2^{\alpha}$.

\begin{equation}\label{m2}\mathcal{L}^1\left(   J_2^{\alpha} \right) = 1-\gamma(\delta) \geq \frac{1}{2}\delta^2.\end{equation}

Now, we can write
\begin{equation}\label{m3}J_2^{\alpha} = g_{\gamma(\alpha),\alpha}\left(J_1^{\alpha}\right)
= g_{\gamma(\alpha),\alpha}\left(A\cap J_1^{\alpha}\right) + g_{\gamma(\alpha),\alpha}\left(A^c\cap J_1^{\alpha}\right).\end{equation}
\\

Putting \eqref{m1}, \eqref{m2}, and \eqref{m3} together, we see that
$$\frac{\mathcal{L}^1\left(  g_{\gamma(\alpha),\alpha}(A\cap J_1^{\alpha}) \right)}{ \mathcal{L}^1\left(  J_2^{\alpha} \right)   }
>  1-  6\epsilon.$$
Combining this and \eqref{m5}
we obtain that \eqref{z97} holds.

\textbf{case 1b: $\alpha \in [0,\delta]$, $c=\gamma(\alpha)+e$, for each $|e| <\frac{\delta^2}{4}$:}\\
Next, if $c=\gamma(\alpha)+e$, for some $|e|<\frac{\delta^2 }{4}$, then $$g_{c,\alpha}\left(J_1^{\alpha}\right)= J_2^{\alpha}+e.$$   It follows that
\begin{align*}
J_2^{\alpha}
&=  \left( J_2^{\alpha}  \cap  g_{c,\alpha}\left(J_1^{\alpha}\right)  \right) \,\bigcup \,\left( J_2^{\alpha} \backslash \left( J_2^{\alpha}+e \right)   \right)\\
&= \left(   J_2^{\alpha}   \cap g_{c,\alpha}\left(A \cap J_1^{\alpha}\right) \right)
\, \bigcup \,\left(  J_2^{\alpha}   \cap  g_{c,\alpha}\left(A^c \cap J_1^{\alpha}\right) \right)
 \,\bigcup \,\left( J_2^{\alpha} \backslash \left( J_2^{\alpha}+e \right)   \right).
 \end{align*}
 \vskip.125in

Putting this together with \eqref{m1} and \eqref{m2}, we see that
\begin{equation}\label{mappedAdensity}
\frac{\mathcal{L}^1 \left( J_2^{\alpha}   \cap  g_{c,\alpha}\left(A \cap J_1^{\alpha}\right) \right)}{\mathcal{L}^1\left(  J_2^{\alpha} \right)   }
>  1-  6\epsilon- \frac{e}{\frac{1}{2}\delta^2}.\end{equation}
We combine this with the observation that, by \eqref{Bdensity},  $$\frac{ \mathcal{L}^1\left( J_2^{\alpha} \cap B \right)}{  \mathcal{L}^1\left(  J_2^{\alpha} \right)   } > 1- \epsilon.$$
Thus, $g_{c,\alpha}\left(A\right) \cap B \neq \emptyset$  holds for $\alpha \in [0,\delta]$ and $c=\gamma(\alpha)+e$ provided that
$\left| e \right|  < \frac{\delta^2(1-7\epsilon)}{2}$.
Set $e_{0}= \frac{\delta^2(1-7\epsilon)}{2}$ and $$N(\gamma(\alpha) ) = \left(\gamma(\alpha) -e_0, \gamma(\alpha) +e_0\right)$$ to complete the proof of the Lemma \ref{K1} in case 1b.
\\

\begin{corollary}\label{EWpoles}
For $\alpha \in [-\delta,\delta]$, a simple modification of the proof of the Lemma \ref{K1} (mainly, reversing the roles of the sets $A$ and $B$ )  implies that the set
$$\left\{(c, \alpha) : c\in N(\pm\gamma(\alpha)) \right\}\subset \left(A\times B \right)+S^1.$$
Moreover,
replacing $g_{c,\alpha}$ by $h_{c,\alpha}$ and reversing the roles of $A$ and $B$ in the proof of Lemma \eqref{K1} yields that $\left\{(c, \alpha) : c\in N(-\gamma(\alpha)) \right\}\subset \left(A\times B \right)+S^1$
\end{corollary}
\vskip.125in

Gathering the results of Corollary's \ref{NSpoles} and \ref{EWpoles}, we conclude that $\left(A\times B \right)+S^1$ contains an open ball about each of the poles of $S^1$: $(1,0), (0,1), (-1,0), (0,-1)$.
\\

\textbf{case 2: $\delta \le  \alpha <1$, for each $|e| <\frac{\delta^2}{4}$:}\\
The previous case relied on an obtaining an upper bound on $\mathcal{L}^1\left( g_{\gamma(\alpha),\alpha}\left(A^c\cap J_1^{\alpha}\right)\right)$ (see \eqref{m1} above)
and a lower bound on $\mathcal{L}^1\left(   J_2^{\alpha} \right) $ (see \eqref{m2} above).
To handle $\alpha $ away from zero, $\delta \le \alpha < 1$, we perform a similar argument where the mean value theorem takes the role of Taylor's theorem above.
 Here are the details:

For $\delta \le \alpha \le 1$, set $J_1=\left[0, \frac{\delta}{2}\right]$ and
$J_2^{\alpha}= \left[ \gamma(\alpha)-\gamma(\alpha-\delta/2 ), 0\right] \subset \left[ \gamma(1)-\gamma(1-\delta/2 ), 0\right]$.

The upper bound in \eqref{m1} still holds when $J_1=\left[0, \frac{\delta}{2}\right]$.  Next, we have by the mean value theorem that there exists $\widehat{\alpha} \in \left(  \alpha-\frac{\delta}{2}, \alpha\right)$ so that
\begin{equation}\mathcal{L}^1\left(   J_2^{\alpha} \right)
=\gamma\left(\alpha-\frac{\delta}{2} \right)  -  \gamma(\alpha) = -\frac{\delta}{2} \cdot \gamma'(\widehat{\alpha})
.\end{equation}
Since the derivative $\gamma'$ is strictly increasing on $(0,1)$ and $ \widehat{\alpha} \in \left(  \alpha-\frac{\delta}{2}, \alpha\right)\subset \left(\frac{\delta}{2},1\right)$, then
\begin{equation}
\mathcal{L}^1\left(   J_2^{\alpha} \right)
>-\frac{\delta}{2} \cdot \gamma'\left( \frac{\delta}{2} \right) \sim \delta^2.
\end{equation}
The proof proceeds as before.
\end{proof}
\subsection{Proof of Lemma \ref{g99}}\label{g57}

 \begin{proof}[Proof of Lemma \ref{g99}] We may assume that $A$ is not a singleton and that $A$ is connected (that is we cannot find $G_1,G_2$ open sets such that $G_i\cap A\ne\emptyset$, $i=1,2$ and $A=(G_1\cap A)\cup(G_2\cap A)$). Otherwise we change to one of its connected component which is not a singleton.
    Let
    $$
    \widehat{A}:
    =
    \left\{
    x\in\mathbb{R}^2:
    \exists a_1,a_2\in A: \|x-a_1\|<1, \|x-a_2\|>1
    \right\}.
    $$
    It is immediate form the definition that $\widehat{A}$ is open and
    using $\# A \ne 1$ we obtain that $\widehat{A}\ne\emptyset$.
  Hence, to verify the Lemma, it is enough to check that
  \begin{equation}\label{094}
   \widehat{A}\subset A+S^1.
  \end{equation}
  To get contradiction, assume that there exists
$x\in\widehat{A}$ such that $x\not\in A+S^1$. Then $ \forall a\in A,\quad
  a\not\in x+S^1$. That is
\begin{equation}\label{093}
 A\cap (x+S)^1=\emptyset.
\end{equation}
  Let
  $$
  G_1:=\left\{y\in\mathbb{R}^2:
  \|y-x\|<1
  \right\},\quad
  G_2:=\left\{y\in\mathbb{R}^2:
  \|y-x\|>1
  \right\}.
  $$
   Clearly, $G_1,G_2$ are open, and it follows from $x\in\widehat{A}$ that
   $G_1\bigcap A \neq \emptyset,G_2\bigcap A\ne\emptyset$. But by \eqref{093} we have
   $$
   A=(G_1\cap A)\cup(G_2\cap A)
   $$
   which contradicts to the assumption that $A$ is connected. This shows that \eqref{094} holds and the fact that $\widehat{A}$ is a non-empty open set yields that
   \eqref{095} holds.
  \end{proof}



 \section{Proof Theorem \ref{g61}}\label{g58'}
Let $t=(t_1, t_2) \in C(1/3):= C_{1/3}\times C_{1/3}.$
We show that $\Delta_{t}(C(1/3))$ contains an interval.
Let $v>0$, and define
$$g_v(x)=t_2+  \sqrt{v^2- (x-t_1)^2}.$$
Observe that if
$$g_v(x)=y \text{\,\,\, for some\,\,\,}  x,y\in C_{1/3},$$
then $v\in \Delta_{(t_1,t_2)}\left(C(1/3)\right).$
\\

We verify that there exists a non-empty open interval $I$, in the domain of $g_v$, with $I\cap C_{1/3}\ne \emptyset $ and a non-empty open interval $V$ such that
\begin{equation}\label{g55}
g_v(I\cap C_{1/3})\cap C_{1/3}\ne \emptyset, \quad \forall v\in V.
\end{equation}
\\

The idea behind the proof of \eqref{g55} is to use a modification of the proof of the Newhouse gap lemma which is presented in Palis and Takens \cite{PT00}.  Given a set $K$, we refer to the connected components of the compliment of $K$ as gaps. The \textit{Newhouse gap lemma} states that two Cantor sets  $K$ and $L$ intersect one another provided that the product of their ``thicknesses'' (see Section \ref{N100})  is greater than one, $K$ is not contained in a gap of $L$, and $L$ is not contained in a gap of $K$.
\\

Since the thickness of $C_{1/3}$ is one and the thickness of $g_v(C)$ is smaller than one,  Newhouse gap  Lemma cannot be applied directly but we find inspiration in its proof.
The aim now is to construct the intervals $I$ and $V$.
\\

By symmetry of the middle-third Cantor set, we need only show that $\Delta_{t}(C(1/3))$ contains an interval for
$t= (t_1,t_2)\in \widetilde{C} \times \widetilde{C}$, where $\widetilde{C}:=    \left[0,\frac{1}{3}\right]\cap C(1/3).$
Fix such a $(t_1,t_2)$ throughout.
\\

Define the set of left and right gap endpoints of $C_{1/3}$:
\begin{equation}\label{goodleft} C_L=\{x\in C_{1/3}:  x \text{ is a left-end-point of a finite gap of } C_{1/3}\},\end{equation}
\begin{equation}\label{goodright} C_R=\{x\in C_{1/3}:  x \text{ is a right-end-point of a finite gap of } C_{1/3}\} .\end{equation}
Clearly, $C_L$ and $C_R$ are dense in $C_{1/3}$.

\begin{lemma}\label{goodpoint}
For sufficiently small $\,\epsilon>0$,
there exists $(u_1, u_2) \in C_{1/3}\times C_{1/3}$ with $u_1>t_1$ and $u_2>t_2$ satisfying:\\
(i) $u_1 \in C_L$ and $u_2 \in C_R$, and \\
(ii) $\left|g_{v_0}'(u_1) \right|\in [1+2\epsilon, 3-2\epsilon]$, where $v_0=\sqrt{ (u_1-t_1)^2 + (u_2-t_2)^2}$\\
(iii) $g'_{v_0}(u_1)<0$.
\end{lemma}

  \begin{proof} [Proof of Lemma \ref{goodpoint}]
Let $\mathcal{R}^{\varepsilon}_{(t_1, t_2)}$ be the union of the  set of half-lines starting
$(t_1,t_2)$ with slope contained $\left[\frac{1}{3}+2\epsilon, \,\,1-2\epsilon\right]$ for a small $\varepsilon>0$.

It follows from elementary geometry that we can find an $\varepsilon>0$ so that the set $\mathcal{R}^{\varepsilon}_{(t_1, t_2)}$  contains an open neighborhood of a point in $C_{1/3}\times C_{1/3}$.
Using that both $C_L$ and $C_R$ are dense in $C_{1/3}$, we obtain that
\begin{equation}\label{g54}
\exists \,\,(u_1,u_2)\in C_L\times C_R \mbox{ such that }\
t_1<u_1, t_2<u_2,\
\frac{u_1-t_1}{u_2-t_2}\in \left[1+2\epsilon, \,\,3-2\epsilon\right].
\end{equation}
Let
\begin{equation}\label{g53}
  v_0:=\sqrt[]{(u_1-t_1)^2+(u_2-t_2)^2}.
\end{equation}
It is straight forward to verify that $g'_{v_0}(u_1)<0$ and
\begin{equation}\label{k96}
  |g'_{v_0}(u_1)| \,\, = \,\,\frac{|u_1-t_1|}{|u_2-t_2|} \,\, \in \,\, \left[1+2\epsilon, \,\,3-2\epsilon\right].
\end{equation}
\end{proof}
For the rest of the proof of Theorem \ref{g61}, fix $\varepsilon>0$, the pair $(u_1,u_2)\in C_L\times C_R$, and the corresponding $v_0$ from Lemma \ref{goodpoint}.

\begin{lemma}\label{K97}
There exists $\delta>0$ so that \newline
(i) $$
  \left|g_{v}'(x) \right|\in [1+\epsilon, 3-\epsilon] \mbox{\,\, if \,\,}
  |x-u_1|<\delta \mbox{ and }|v-v_0|<\delta.
$$
 Moreover, there exists $\delta'>0$ so that if $v_0 \le v < v_0 +\delta'$, then
  \newline
(ii)  $g_{v}([u_1-\delta,u_1]\cap C_{1/3})$ is not contained in a gap of $C_{1/3}$ and $C_{1/3}$ is not contained in a gap of\  $\ g_{v}([u_1-\delta,u_1]\cap C_{1/3})$.

\end{lemma}

\begin{proof}[Proof of Lemma \ref{K97}]
Item (i) is an immediate consequence of the continuity of the derivative of the function $(u,v)\mapsto g_v(x)$ at the point $(u_1,v_0)$.
Item (ii) holds when $v=v_0$ because $u_2=g_{v_0}(u_1)$.
Now we consider the case when $v\in (v_0,v_0+\delta)$.

If $v>v_0$, then $ g_{v_0}(u_1) < g_{v}(u_1)$.  Choose $v>v_0$ small enough that
\begin{equation}\label{blah}  g_{v_0}(u_1) < g_{v}(u_1) < g_{v_0}(u_1) + \delta.\end{equation}
In particular, choose $0<\delta' \le \delta$ so that if $v_0 \le v < v_0 +\delta'$, then \eqref{blah} holds.
Combining this with the observation that, by the lower bound on the derivative from (i),
$g_v(u_1-\delta) - g_v(u_1) >\delta$,  we have
$$u_2 = g_{v_0}(u_1) < g_{v}(u_1) < u_2 + \delta<  g_{v}(u_1- \delta).$$

We now argue that $\delta>0$ can be chosen apriori in such a way that both $u_1- \delta$ and $u_2+\delta$ are in $C_{1/3}$  by setting $\delta = \frac{1}{3^N}$ for $N$ sufficiently large.
The choice $(u_1,u_2)\in C_L\times C_R$ implies that both
$\left[u_1-\frac{1}{3^N},u_1\right]$ and
$\left[u_2,u_2+\frac{1}{3^{N}}\right]$ are cylinder intervals of $C_{1/3}$ if $N$ is sufficiently large.

\end{proof}

We fix $\delta, \delta'>0$ which satisfy Lemma \ref{K97} and
we define:
\begin{equation}\label{g50}
  I:=\left(u_1-\delta,u_1\right),
  V:=(v,v+\delta'), \mbox{ \,\,\,and }\quad C:=I\cap C_{1/3}.
\end{equation}

Further we fix an arbitrary $v\in V$ and we define
$$
g:C\to \mathbb{R},\qquad  g(x):=g_v(x).
$$
As we discussed above (see \eqref{g55}), to prove the theorem  it is enough to prove
that
\begin{equation}\label{g49}
  g(C)\cap C_{1/3}\ne \emptyset.
\end{equation}
\begin{definition}\label{g56}

\begin{description}
  \item[(a)] Let $U$ be a gap of $C_{1/3}$.
  We define the \textit{right-bridge}  of $U$ as the minimal distance between $U_r$ and the left-end-point of another gap $U'$, so that $U'_l>U_r$ and $|U'|\geq |U|$.  The \textit{left-bridge}  of $U$ is defined analogously.
  \item[(b)] Let $U$ be a bounded gap of $C_{1/3}$ and $U'$ be a bounded gap of $C$.
 We call $\left(U, g(U')\right)$ a \textit{gap pair}
 if $U$ contains exactly one end-point of $g(U')$, and $g(U')$ contains exactly one end-point of $U$.
\end{description}

\end{definition}

It is an easy exercise to check that the following Fact holds:

\begin{fact}\label{g48}
  The bridges of $g(C)$ are the $g$-images of the bridges of the Cantor set $C$.
\end{fact}
The following plays a central role in the proof of our theorem:
\begin{lemma}\label{PTsub}
Let $\left(U, g(U')\right)$ be an arbitrary gap pair. The right-end-point of $U$, is denoted by $U_r$ an d the left-end-point of $U$, denoted by $U_l$.
 Let $B_r$ denote the right-bridge of $U$ and $B_l$ denote the left-bridge.
Let $B_r^g$ denote the right-bridge of $g(U')$ and $B_l^g$ denote the left-bridge of $g(U')$. Then at least one of the following hold: \\
(a) $U_r$, is contained in $g(U')$, and $|B_r|\geq |g(U')|$.\\
(b)   $U_r$, is contained in $g(U')$, and $|U|\le |B_l^g|$.\\
(a)  $U_l$, is contained in $g(U')$, and $|B_l|\geq |g(U')|$.\\
(b)  $U_l$, is contained in $g(U')$, and $|U|\le |B_r^g|$.
\end{lemma}
 Observe that Lemma \ref{K97} guarantees  the existence of a gap $U'$ of $C$ and a gap $U$ of $C_{1/3}$ such that $(U,g(U'))$ form a gap pair.
\begin{proof}[Proof of Lemma \ref{PTsub}] For symmetry, without loss of generality we may assume that $U_r\in g(U')$.
It is immediate from the construction of the middle-third Cantor set that
$
|B_l| =|U| \mbox{ and }
|B_r|= |U|.
$
Observe that
\begin{description}
  \item[(i)] If $|U'| < |U|$, then $|U'|  \leq \frac{1}{3} |U|$ and so, by the Mean Value Theorem and  \eqref{k96} we have $|g(U')| < |U|$.
  \item[(ii)] If $|U'| \geq |U|$, then $|B_l^g|>|U|$ and $|B_r^g|>|U|$. To see this, remember that the bridges of $g(C)$ are the $g$-images of the bridges of $C$ and the bridges in $C$ adjacent to $U'$ have the same length as $U'$. We apply $g$ for the bridges adjacent to $U'$ to get the bridges $B_{l}^{g}$
      and $B_{r}^{g}$. Now we use
       the Mean Value Theorem and  \eqref{k96} to verify the assertion above.
\end{description}
\begin{figure}
  \centering
  \includegraphics[width=12cm]{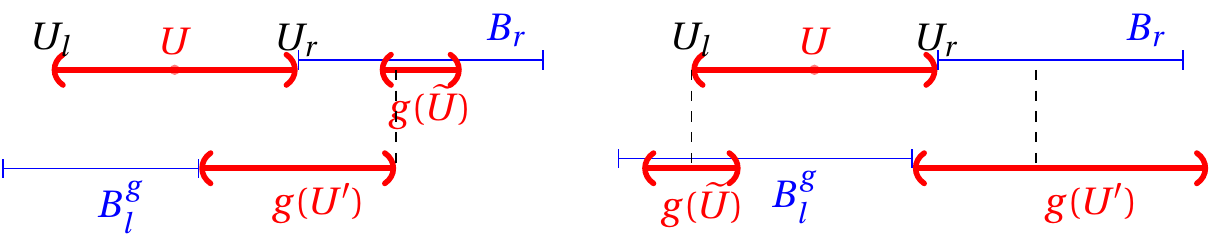}
  \caption{$|\widetilde{U}|<|U|$ and we can replace $(U, g(U'))$ with $(\widetilde{U},g(U'))$ since $(\widetilde{U},g(U'))$ is also a gap pair}\label{g47}
\end{figure}

It is immediate that (i) implies that (a) holds and (ii) implies that
(b) holds.
\end{proof}

Now we can finish the proof of Theorem \ref{g61} as follows:
From now we proceed as the proof of the Gap Lemma
in \cite{PT00}. Namely, Lemma \ref{g48} guarantees that
for any gap pair $(U, g(U'))$ we can replace
either $U$ with a shorter gap $U_1$ of $C_{1/3}$
or we can find a gap $U'_1$ of  $C$ such that $|g(U'_1)|<|g(U')|$.
Since the total length of all gaps is summable, after an infinite sequence of these replacements we get a sequence of gaps with length convergent to zero. Since the closure of these gaps contains points from both of $C_{1/3}$ and $g(C)$ we obtain by a usual compactness argument  that
 \eqref{g49} holds and this completes the proof of the theorem.


\section{proof of Theorem \ref{k1}}\label{sectionproofk1}
\begin{proof}
Recall that we set $C(1/3):=C_{1/3}\times C_{1/3}$ and  $\gamma(x) = \sqrt{1-x^2}$.
As in the proof of Theorem \ref{main}, we fix $\alpha \in (0,1)$ and reduce matters to showing that
\begin{equation}\label{k1set}\left\{y - \gamma(\alpha-x) : x,y \in C_{1/3}\mbox{ and } |\alpha-x|\le1\right \}\end{equation}
contains a non-empty open interval.
Namely, if  $t= y - \gamma(x-\alpha)$ for some $x,y \in C_{1/3}$  then
\begin{equation}\label{z96}
  (\alpha,t)\in C(1/3)+S^1.
\end{equation}
Therefore to verify that $\left(C(1/3)+S^1\right)^{\circ}\ne  \emptyset $ it is enough to show that there exists an interval $\Lambda$ of $\alpha$s so that
\begin{equation}\label{k2}\bigcap_{\alpha \in \Lambda} \left\{y - \gamma(\alpha-x) : x,y \in C_{1/3}\mbox{ and } |\alpha-x|\le1\right \}\end{equation} contains a non-degenerate interval.
The conclusion of the theorem follows immediately upon establishing that the set in \eqref{k2} contains a non-empty open interval; the details of this reduction are explained in the proof of Theorem \ref{main}.
\\

The remainder of the proof is dedicated to establishing that the set in \eqref{k2} contains a non-empty open interval.  The proof follows a similar outline to that of Theorem \ref{g61} where we study pinned distance sets of $C(1/3)$.
\\

Fix $\alpha \in (0,1)$ and a scalar $t$.  Set $$h_{t,\alpha}(x) = t-\gamma(\alpha-x).$$
As explained above, we need to prove that there exists an interval of $\alpha$s and $t$s such that
 \begin{equation}\label{z95}
  h_{t,\alpha}(C_{1/3}') \cap C_{1/3} \neq \emptyset,
\end{equation}
where $C_{1/3}'$ is a suitable restriction of $C_{1/3}$ to the domain of $h_{t,\alpha}$.
\\

Let $C_L$ be as in \eqref{goodleft} respectively.
Recall that $C_L$  is dense in $C_{1/3}$.

\begin{lemma}\label{goodpoint'}
For sufficiently small $\,\epsilon>0$,
there exists a scalar $\alpha_0$, a point pair
$(u_1, u_2) \in C_{1/3}\times C_{1/3}$ with $|u_1-\alpha|<1$, a scalar $t_0$, and an $N\in \mathbb{N}$ satisfying:\\
(i) $u_1 \in C_L$ and $u_2 \in C_L$,\\
(ii) $h_{t_0, \alpha_0}'(u_1) \in [1+2\epsilon, 3-2\epsilon]$, and\\
(iii) $h_{t, \alpha}'(x) \in [1+\epsilon, 3-\epsilon]$ whenever $max\{|x-u_1|, |\alpha-\alpha_0|, |t-t_0|\}\le \frac{1}{3^{N}}$.
\end{lemma}
\begin{proof}
Observe that for $t$ arbitrary,
$$h_{t, \alpha}'(x) = \gamma'(\alpha-x)= \frac{(x-\alpha)}{\sqrt{1-(x- \alpha )^2}}$$
is a bijection from $(\alpha,\alpha+1)$ to $(0,\infty)$.  Let $I_{\alpha}\subset (\alpha, \alpha+1)$ denote the open set of $x$ so that
$h'_{t,\alpha}(x)\in (1+2\epsilon, 3-2\epsilon)$.  Observe that the length of $I_{\alpha}$ is independent of $\alpha$.
\\

Choose $\alpha_0$ so that $I_{\alpha_0}$ intersects $C_{1/3}$.
Choose $u_1 \in I_{\alpha_0}\cap C_L$.  Next, choose $t_0$ so that
$$u_2:=h_{t_0, \alpha_0}(u_1) \in C_L.$$

It follows by the continuity of $H(t, \alpha, x):= h'_{t,\alpha}(x)$ at $(t_0,\alpha_0, u_1)$, that there exists a non-empty open neighborhood of $(t_0,\alpha_0, u_1)$ on which $h'_{t,\alpha}(x)\in (1+\epsilon, 3-\epsilon)$.  Choose $N\in \mathbb{N}$ so that $\frac{1}{3^N}$ is strictly less than the radius of this interval.
\end{proof}
Fix the points $u_1, u_2, \alpha_0, t_0$ and $\delta>0$ as in Lemma \ref{goodpoint'}.

\begin{lemma}\label{K97a}
There exists $\delta'>0$ and $M\in \mathbb{N}$ so that if $\alpha_0< \alpha<\alpha_0+\delta'$ and $t_0-\delta'<t<t_0$, then
$h_{t,\alpha}([u_1-\frac{1}{3^M},u_1]\cap C_{1/3})$ is not contained in a gap of $C_{1/3}$ and $C_{1/3}$ is not contained in a gap of\  $h_{t,\alpha}([u_1-\frac{1}{3^M},u_1]\cap C_{1/3})$.
\end{lemma}
\begin{proof}
Let $M\in \mathbb{N}$ to be determined.
Use the continuity of $H(t,\alpha):= h_{t,\alpha}(u_1)$ at $(t_0, \alpha_0)$ to
choose $\delta'>0$ ($\delta'\le\delta$, where $\delta$ is as in Lemma \ref{goodpoint'}) so that if $\max\{|\alpha-\alpha_0|, |t-t_0|\}<\delta'$, then
$$u_2-\frac{1}{3^M}< h_{t,\alpha}(u_1) <u_2,$$
where $u_2= h_{t_0,\alpha_0}(u_1).$
Next, use the derivative assumption proved in the previous lemma to prove that whenever
$\max\left\{|\alpha-\alpha_0|, |t-t_0|\right\}<\frac{1}{3^M}$ and $M\geq N$ (see Lemma \ref{goodpoint'} part (iii)), then
$$h_{t,\alpha}(u_1) - h_{t,\alpha}\left(u_1-\frac{1}{3^M} \right)> \frac{1}{3^M}.$$
It follows that if $\max\{|\alpha-\alpha_0|, |t-t_0|\}<  \frac{1}{3^M} $ and $\frac{1}{3^M} \le \delta'$, then
$$h_{t,\alpha}\left(u_1-\frac{1}{3^M} \right) < u_2 - \frac{1}{3^M} < h_{t,\alpha}(u_1)  < u_2.$$
The choice $(u_1,u_2)\in C_L\times C_L$ implies that both
$\left[u_1-\frac{1}{3^M},u_1\right]$ and
$\left[u_2-\frac{1}{3^{M}}, u_2 \right]$ are cylinder intervals of $C_{1/3}$ if $M$ is chosen sufficiently large.
Choosing such an $M$ completes the proof of the lemma.
\end{proof}
Lemma \ref{K97a} guarantees the existence of gap pairs for $h_{t,\alpha}([u_1-\frac{1}{3^M},u_1]\cap C_{1/3})$ and $C_{1/3}$ whenever
$\alpha_0< \alpha<\alpha_0+\delta'$, $t_0-\delta'<t<t_0$, and $\delta=\frac{1}{3^M}$ for $M$ sufficiently large.
\\

The proof of the Theorem now proceeds exactly as in Theorem  \ref{g61} where it was established that the existence of a gap pair between two sets guarantees their intersection.
In particular, see the proof of Theorem \ref{g61} from equation \eqref{g49} onward where we set
$g(x):= h_{t,\alpha}(x).$
\end{proof}

 \section{Proof of Theorems \ref{g92} and  \ref{g71}}\label{p99}
 Throughout this Section we use the following notation:
\begin{definition}\label{g95}
Let $\Gamma$ be a polygon of $n$ sides. Then
\begin{equation}\label{077}
  \Gamma=\bigcup_{i=1}^{n}I_i,
\end{equation}
 where $I_i$
is a straight line segment.
We write $\widetilde{\ell }_i$ to denote the straight line which contains $I_i$, and $\ell _i$ for the straight line through the origin which is parallel to $\widetilde{\ell }_i$.
Let $\alpha_i$ denote the angle between $\ell _i$ and the $x-$axis.
\\

For each $i\in \{1,\dots, n\}$, choose $u_i\in\mathbb{R}$ such that
\begin{equation}\label{g87}
  \widetilde{\ell }_i=u_i \cdot \mathbf{e}_{\alpha_{i}^{\perp}}+\ell _i,
\end{equation}
where $\mathbf{e}_\alpha\in S^1$ denotes the unit vector of  angle $\alpha$ and $<e_{\alpha},\, e_{\alpha^{\perp}}> =0$.
\\


We now introduce a dense family of parallel lines. Let $G \subset \mathbb{R}$ and set
\begin{equation}\label{g84}
  P_i(G)
  :=\bigcup\limits_{g\in G}\left(g \cdot \mathbf{e}_{\alpha_{i}^{\perp}}+\ell _i\right)
\mbox{ and }
\widetilde{P}_i(G):=P_i(G+u_i)=
P_i(G)+u_i\mathbf{e}_{\alpha_{i}^{\perp}}.
\end{equation}

Using that $\ell _i+\ell _i=\ell _i$  we obtain that
\begin{equation}\label{g85}
P_{i}(G)+\widetilde{\ell }_i = \widetilde{P}_{i}(G)
\mbox{ and }
P_{i}(G)^c+\widetilde{\ell }_i = \widetilde{P}_{i}(G)^c.
\end{equation}

\end{definition}

\begin{proof}[Proof of Theorem \ref{g92}] Let $\Gamma$ be a polygon of $n$ sides. We use the notation of definition \ref{g95}.
Fix a $G \subset \mathbb{R}$ which is a dense $G_\delta$ set with
$\mathcal{L}^1(G)=0$.
With this choice of $G$, both $P_i(G)$ and $\widetilde{P}_i(G)$ defined in \eqref{g84} are  dense $G_\delta$ subsets of $\mathbb{R}^2$ of zero $\mathcal{L}^2$-measure for every $i=1, \dots ,n$.
We define
$$
A:=\bigcap\limits_{k=1}^{n} \left(P_{k}(G)\right)^{c},
$$
Observe that $A$ is of full measure.

By the definition of $\widetilde{\ell }_i $ we have
$$
A+\Gamma \subset
A+\bigcup\limits_{i=1}^{n} \widetilde{\ell }_i=
\bigcup\limits_{i=1}^{n}\left(A+\widetilde{\ell }_i\right)
$$
It remains to verify that $\left(\bigcup\limits_{i=1}^{n}\left(A+\widetilde{\ell }_i\right)\right)^{\circ}=\emptyset$.
This is so because, by \eqref{g85} and the definition of $A$, we have that for each $i$:
\begin{equation}
A+\widetilde{\ell _i}
\, \subset \,\left( \widetilde{P}_i(G) \right)^c,
\end{equation}
which implies that
\begin{equation}
\bigcup_{i=1}^{n}\left( A+\widetilde{l_i} \right)
\, \subset \, \bigcup_{i=1}^{n}\left(  \widetilde{P}_i(G) \right)^c.
\end{equation}

That is,
\begin{equation}
\left(A+\Gamma \right)^c
\, \supset \, \bigcap_{i=1}^n \widetilde{P}_i(G).
\end{equation}

The set on the right hand side is dense since it is a finite intersection of dense $G_\delta$ sets (see the Baire category theorem). This proves that $A+\Gamma$ is disjoint from a dense set, and so its interior must be empty.
\bigskip

\end{proof}

 \begin{proof}[Proof of Theorem \ref{g71}]
  Besicovitch proved (see \cite[Theorem 11.1]{Mat15}) that there exists a Borel set $\widetilde{B}\subset \mathbb{R}^2$ such that
$\widetilde{B}$ contains a line  in every direction but $\mathcal{L}^2(\widetilde{B})=0$.
Following Mattila \cite{Mat15} we define
\begin{equation}\label{g73}
  B:=\bigcup\limits_{r\in\mathbb{Q}\times\mathbb{Q}} (r+\widetilde{B}).
\end{equation}
Then $\mathcal{L}^2(B)=0$ and in every direction there is a dense set of lines contained in $B$. Now let
\begin{equation}\label{g72}
A:=D\cap B^c,
\end{equation}
where $D$ is the open unit disc $D:=\left\{(x,y):x^2+y^2 < 1\right\}$.
 Then $A$ is a set full Lebesgue measure in $D$ but in every direction  there is a dense set of straight lines which do not intersect $A$.
 \\

 Fix an arbitrary $\theta\in [0,\pi)$. Let $N$ be the perimeter of $[-1,1]^2$.
 Recall that we defined $N_\theta$ as the rotated (with angle $\theta$ in anti-clockwise direction around the origin ) image of $N$. Let $N_{\theta}^{N}$, $N_{\theta}^{W}$, $N_{\theta}^{S}$, $N_{\theta}^{E}$ be the rotated (with the same rotation as above) image of the Northern, Western, Southern and Eastern wall of $[-1,1]^2$ respectively. Let
 $$
 A_\theta:=\mathrm{proj}_{\theta}(A).
 $$
 Then using that $A$ is contained in the unit disk we obtain that
 \begin{equation}\label{g69}
   ((A_\theta+1)\times\ell _{\theta}^\perp)\cap ((A_\theta-1)\times \ell _{\theta}^\perp)=\emptyset.
 \end{equation}
 Here we considered $A_\theta\pm 1$ as subsets of $\ell _\theta$ and the product is meant in the
 $\ell _\theta, \ell _{\theta^\perp}$ coordinate system.
 By the definition we have
 \begin{equation}\label{g68}
   (A_\theta+1)^{\circ}=\emptyset \mbox{\,\, and \,\, }
   (A_\theta-1)^{\circ}=\emptyset.
 \end{equation}
 Putting together \eqref{g69}, \eqref{g68} and the fact that
 \begin{equation}\label{g67}
   A+N_{\theta}^{E}\subset (A_\theta+1)\times\ell _{\theta}^\perp \mbox{\,\, and \,\, }
   A+N_{\theta}^{W}\subset (A_\theta-1)\times\ell _{\theta}^\perp,
 \end{equation}
  we obtain that there is a dense set of lines parallel to $\ell _{\theta^\perp}$ which is in the complement of
  $$
 \left(A+N_{\theta}^{E}\right)\cup\left(A+N_{\theta}^{W}
  \right)
  $$
 Exactly in the same way one can verify that there is a dense set of lines parallel to $\ell _{\theta}$ which is in the complement of
  $$
   \left(A+N_{\theta}^{S}\right)\cup\left(A+N_{\theta}^{N}
  \right)
  $$
 Observe that the intersection of the two dense families of lines mentioned above (one parallel to $\ell _\theta$ the other one is parallel to $\ell _\theta^\perp$) intersect each other in a dense set $E$. Clearly $(A+N_\theta)\cap E=\emptyset$ which completes the proof of the theorem.
\end{proof}


\end{document}